\documentclass[12pt]{article}
\usepackage{graphicx}
\usepackage{amsmath,amsthm,bm,mathrsfs}
\usepackage{subcaption}
\usepackage{enumerate}
\usepackage{amssymb}
\usepackage{amsxtra}
\usepackage{latexsym}
\newtheorem{theorem}{Theorem}%[section]

\newtheorem{lemma}{Lemma}[section]

\newtheorem{example}{Example}[section]

\numberwithin{equation}{section}

\numberwithin{equation}{section}

\begin{document}
	
	\title{Existence and stability analysis of solutions for a ultradian glucocorticoid rhythmicity and acute stress model.}
	
	\author{ {\sc Casey Johnson}\\[2pt]{\small Institute of Mathematical Sciences, Claremont Graduate University, Claremont, USA,}\\
		[6pt]
		{\sc Roman M. Taranets}\\[2pt]
		{\small Institute of Applied Mathematics and Mechanics of the NASU,  Sloviansk, Ukraine,}\\
		[6pt]
		{\sc Natalia Vasylyeva}\\[2pt]
		{\small Institute of Applied Mathematics and Mechanics of the NASU,  Sloviansk, Ukraine,} \\[6pt]
		{\sc and} \\ [6pt]
		{\sc Marina Chugunova} \\ [2pt]
		{\small Institute of Mathematical Sciences, Claremont Graduate University, Claremont, USA,}\\
		[6pt]
		{\rm [Received on ************]}\vspace*{6pt}}
	\pagestyle{headings}
	\markboth{C. JOHNSON, R. M. TARANETS, M. CHUGUNOVA, N. VASYLYEVA}{\rm EXISTANCE AND STABILITY ANALYSIS OF HPA AXIS MODEL}
	\maketitle
	
	\maketitle
	\begin{abstract}
		{The hypothalamic pituitary adrenal (HPA) axis responds to physical and mental challenge to maintain homeostasis in part by controlling the body's cortisol level. Dysregulation of the HPA axis is implicated in numerous stress-related diseases. 
			For a structured model of the HPA axis that includes the glucocorticoid receptor but does not take into account the system response delay, we analyze linear and non-linear stability of stationary solutions. For a second mathematical model that describes the mechanism of the HPA axis self-regulatory activities and takes into account a delay of system response, we prove existence of periodic solutions under certain assumptions on ranges of parameter values and analyze stability of these solutions with respect to the time delay value.}
	\end{abstract}
	
	\section{Introduction}
	
	Hormones control a vast array of bodily functions, including sexual reproduction and sexual development, whole-body metabolism, blood glucose levels and so on (see \cite{Walker}). Hormones are produced, in main, and released from diverse places including the hypothalamus,
	pituitary, and the adrenal gland. Hormones are capable of a diffusion whole-body effect, as well as a localized effect, depending on the distance between the production site and the site of action. In many ways, the endocrine system is similar to the nervous system, in that it is an intercellular signaling system in which cells communicate via cellular secretions. Further, the distance between the sites of hormone production and action, and the complexities inherent in the mode of transport, make it extraordinarily difficult to construct quantitative models of hormonal control.

	The hypothalamus pituitary adrenal axis is a central neuroendocrine system, which consists of the hypothalamus, pituitary, and adrenal glands. The paraventricular nucleus of the hypothalamus secrets corticotropin releasing hormone (CRH), which is transferred to the pituitary and stimulates the synthesis and release of adrenocorticotropic hormone (ACTH). ACTH moves through the bloodstream and reaches the adrenal gland in which it stimulates the secretion of cortisol. In response to stress, the concentrations of the HPA axis hormones are increased.

	Disruption of HPA axis regulation is known to contribute to a number of stress-related disorders. For example, increased cortisol has been shown in patients with major depressive disorder (see \cite{Juruena, Gold}), and decreased cortisol has been observed in people with post-traumatic stress disorder (see \cite{Rohleder}).

	Multiple models of the hypothalamus-pituitary-adrenal (HPA) axis have been developed to characterize the oscillations seen in the hormone concentrations and to examine HPA axis dysfunction. Most of these models have been constructed using deterministic coupled ordinary differential equations (see \cite{Gonzalez}). A major inconsistency among different existing HPA models  that was mentioned in \cite{Hosseinichimeh} is related to their  treatment of  the circadian and ultradian oscillations. For example, the authors of \cite{Sriram} and \cite{Bairagi} assumed that both oscillations can be generated inside the HPA axis system by interaction of its    elements;  the authors of \cite{Vinther}, \cite {Andersen}, and \cite{Jelic} treat the circadian and ultradian oscillations differently assuming that only ultradian oscillations are HPA axis based but at the same time  circadian rhythms are due to external input. Only one model made no explicit assumption about the origin of the oscillations and was developed to replicate the HPA axis response to CRH injection (see \cite{Conrad}). It has also been suggested that the ultradian rhythm arises from the introduction of a time delay (see \cite{Bairagi}). Other models based on delay-differential equations  include  \cite{Lenbury} and \cite{Gupta}.
	
	To determine if delay-differential equations could predict the general features of cortisol production, the experimental data was compared to a simulated cortisol curve in \cite{Gupta}. Experimental fitting of ACTH for the model was not possible since hypothalamic derived CRH cannot be measured.

	Inclusion of the glucocorticoid receptor in a hypothalamic pituitary adrenal axis model reveals 'bi-stability' (see \cite{Gupta}). To be more concrete, there arises a nonlinear Gauss type function with compact support, which is characterized by the parameter $p_4$. This (Hill function) arises as a result of 'inner' nonlinearity in the physiological system which is produced by the stress impulse, which is activated by the outer impulse that is called by an acute stress. This situation is provided formally by the two parameters $p_4$ and $CRH$. The amplitude of the Hill function determines glucocorticoid receptor $GR$ density in the pituitary, which is coupled nonlinearly in reaction with regulated levels of  $CORT$, which in turn mediate a wide range of physiological processes, including metabolic, immunological and cognitive function (see \cite{Rankin1,McEwen}).

	The stress response is subserved by the stress system which is located both in the central nervous system and the periphery. The principal effects to the stress system include the corticotropin-releasing hormone $CRH$. The secretion of $CRH$ causes the anterior pituitary to synthesize adrenocorticotropin $ACTH$ which then stimulates the adrenal glands to release cortisol that regulate the blood concentration of $CRH$ and $ACTH$ via different negative feedback mechanisms. The $HPA$ axis is the subject of intensive research in endocrinology. This model is based on the feed-forward and feedback interactions between the anterior pituitary and adrenal glands. Because responsiveness of the stress system to stressors is crucial for life, it is important to consider the simpler case when distributions of hormones in the system become unstable by action on stress, and further to consider influence on the delay time as response of the physiological system on action on stress.
	
	Mathematically, it means that we can consider two mathematical models: the first one is described by a system of ordinary differential equations with initial distributions of hormones at a point $t=0$, and the second one is based upon a system of differential equations with initial distributions of hormones on the interval $[-\tau, 0 )$, where $\tau$ is a time delay. It turns out that bi-stability is present in both models, i.e limit distributions of hormones may be stable or unstable depending on parameter values. In the model with a time delay, periodic solutions arise for special distributions of hormones when there is a connection between the concentrations of hormones at end points. Thus, the initial distributions of hormones must be coupled in a special manner. This condition may be considered as a 'normal' reaction of the organism on action of stress. 
	
	In this paper, we  study a system of delay differential equations (see \cite{Gupta,Rankin1}):
	\begin{equation}\label{a-1}
	\frac{da}{dt} = \frac{CRH}{1+p_2 or} - p_3 a =: f_1,
	\end{equation}
	\begin{equation}\label{a-2}
	\frac{dr}{dt} = \frac{(or)^2}{p_4 + (or)^2} +  p_5 - p_6 r =: f_2,
	\end{equation}
	\begin{equation}\label{a-3}
	\frac{do}{dt} = a(t-\tau) - o =: f_3
	\end{equation}
	with initial conditions
	\begin{equation}\label{in-con}
	a(t) = a_{\tau}(t)  \ \forall \, t \in [-\tau,0], \  r(0) = r_0, \ o(0) = o_0,
	\end{equation}
	where
	\begin{equation}\label{smo-con}
	0 \leqslant a_{\tau}(t) \in C^1[-\tau,0], \  r_0 > 0, \  o_0 > 0.
	\end{equation}
	Based on the principles of mass action kinetics, these equations
	describe the production and degradation of the hormones $ACTH$ ($a$), i.\,e. adrenocorticotropin,
	and $CORT$ ($o$), i.\,e. cortisol in humans and corticosterone in rodents, as well as glucocorticoid receptor $GR$ density
	($r$) in the pituitary. Here, $CRH$ is corticotrophin-releasing hormone,
	the parameters $p_{2-6}$ represent dimensionless forms of rate
	constants of the system, and the dimensionless parameter $\tau$
	represents a discrete delay, which accounts for the delayed
	response of the adrenal gland to $ACTH$. The
	dimensionless time $t =0$ corresponds to the maximal value of an
	$ACTH$ pulse.

	The paper is organized as follows. In Section~\ref{sec-2}
	we study the stability of the system (\ref{a-1})--(\ref{a-3}) without delay
	for the given initial distributions of hormones. In Section~\ref{sec-3}
	we analyze the system (\ref{a-1})--(\ref{a-3}) with delay in terms of stability, solvability, and existence
	of periodic solutions.  Moreover, the existence of $\tau$-periodic solutions is proved.
	The system without delay is always asymptotically stable for strictly positive parameter values. When delay is considered, the system remains unstable after a certain delay value. The theoretical results are then compared with numerical simulations.

	\section{Stability analysis of the model without time delay }\label{sec-2}
	
	We consider the following nonlinear ODEs without delay:
	\begin{equation}\label{eq-1}
	\frac{da}{dt} = \frac{A}{1 + p_2 or} - p_3 a,
	\end{equation}
	\begin{equation}\label{eq-2}
	\frac{dr}{dt} = \frac{(or)^2}{p_4 +  (or)^2} + p_5 - p_6 r,
	\end{equation}
	\begin{equation}\label{eq-3}
	\frac{do}{dt} = a - o,
	\end{equation}
	with initial conditions
	\begin{equation}\label{in-con-nnn}
	a(0) = a_0 \geqslant 0, \  r(0) = r_0 > 0, \ o(0) = o_0 > 0,
	\end{equation}
	where  $A := CRH \geqslant 0$, and $p_i \geqslant 0$. Using  Picard's iteration method,
	we can show existence of a unique global in time  non-negative solution. The proof is similar to the
	one of Theorem~\ref{Th-exist}.

	\begin{lemma}\label{L-sss}
		Assume that $A > 0$ and $p_i > 0$. Then the system (\ref{eq-1})--(\ref{eq-3}) has a unique fixed point and this point is asymptotically stable.
	\end{lemma}
	
	\begin{proof}[Proof of Lemma~\ref{L-sss}]
		
		By (\ref{eq-1})--(\ref{eq-3}) we obtain the following equations for the nullclines:
		\begin{equation}\label{eq-4}
		o = a, \ \  \tfrac{A}{1 + p_2 or} - p_3 a = 0, \ \ \tfrac{(or)^2}{p_4 +  (or)^2} + p_5 - p_6 r = 0.
		\end{equation}
		The algebraic system (\ref{eq-4}) has a nonnegative solution in the following domain:
		\begin{equation}\label{con-1}
		D:= \{  (a,r,o) \in R^3_+ :  a =o,  \, 0 \leqslant a \leqslant \tfrac{A}{p_3}, \, \tfrac{p_5}{p_6} \leqslant r \leqslant \tfrac{ p_5 +1}{p_6}  \}.
		\end{equation}
		From (\ref{eq-4}) we have
		\begin{equation}\label{eq-5}
		o = a  = \tfrac{1}{2p_2 r} \biggl[ \sqrt{ 1 + \tfrac{4p_2 A}{p_3}\,r} - 1 \biggr],
		\end{equation}
		\begin{equation}\label{eq-6}
		\tfrac{1}{4 p_2^2} \biggl[ \sqrt{ 1 + \tfrac{4p_2 A}{p_3}\,r} - 1 \biggr]^2 = \tfrac{p_4(p_6 r - p_5)}{1 + p_5 - p_6 r},
		\quad \mbox{i.\,e.} \quad
		A = \tfrac{p_3}{r} \sqrt{\tfrac{p_4(p_6 r - p_5)}{1 + p_5 - p_6 r}}
		\Bigl( 1 + p_2 \sqrt{\tfrac{p_4(p_6 r - p_5)}{1 + p_5 - p_6 r}}  \Bigr).
		\end{equation}
		
		Note that the equation (\ref{eq-6}) has a unique solution $r^* \in (\frac{p_5}{p_6}, \frac{p_5 + 1}{p_6})$ in $D$.
		Really, the function $f_1(r) := \frac{1}{4 p_2^2} \biggl[ \sqrt{ 1 + \frac{4p_2 A}{p_3}\,r} - 1 \biggr]^2$ is
		nonnegative and monotone increasing on $[0,+\infty)$ such that $f_1(0) = 0$, $f_1(+\infty) = +\infty$,
		but the function $f_2(r) := \frac{p_4(p_6 r - p_5)}{1 + p_5 - p_6 r}$ is
		nonnegative and monotone increasing on $[\frac{p_5}{p_6}, \frac{p_5 + 1}{p_6}]$ such that
		$f''_2(r) > 0$, $f_2(\frac{p_5}{p_6}) = 0$,
		$f_2(\frac{p_5 + 1}{p_6}) = +\infty$. Therefore, there is only one intersection of $f_1(r)$ and
		$f_2(r)$ on the interval $[\frac{p_5}{p_6}, \frac{p_5 + 1}{p_6}]$. On the other hand, let us denote by
		$$
		z :=  \sqrt{ 1 + \tfrac{4p_2 A}{p_3}\,r} - 1 \geqslant 0 \Rightarrow r = \tfrac{p_3}{4p_2 A} z (z+2).
		$$
		Then (\ref{eq-6}) can be rewritten in the following form
		\begin{equation}\label{eq-7}
		z^4  + 2 z^3 + C_1 z^2 + C_2 z - C_3 = 0,
		\end{equation}
		where
		$$
		C_1 : = \tfrac{4p_2(p_2 p_3 p_4 p_6 - A(p_5 + 1))}{p_3 p_6}, \ C_2 := 8 p_2^2 p_4 > 0, \
		C_3 := \tfrac{16 A  p_2^3 p_4 p_5}{p_3 p_6} > 0.
		$$
		So, we can find the explicit value of $r^* \in \left(\frac{p_5}{p_6}, \frac{p_5 + 1}{p_6}\right)$
		as a solution of (\ref{eq-7}).
		
		As a result, the system (\ref{eq-1})--(\ref{eq-3})  has only one fixed point $(a^*, r^*, o^*)$ in $D$.
		Here,
		$$
		a^* = o^* = \tfrac{1}{2p_2 r^*} \biggl[ \sqrt{ 1 + \tfrac{4p_2 A}{p_3}\,r^*} - 1 \biggr] =
		\tfrac{1}{r^*}  \sqrt{\tfrac{p_4 (p_6 r^* - p_5)}{1 + p_5 - p_6 r^*}},
		$$
		and $r^*$ is the solution of (\ref{eq-6}) or (\ref{eq-7}).
		
		Next, we find the Jacobian matrix $J^*$ for (\ref{eq-1})--(\ref{eq-3}) at the
		fixed point $(a^*, r^*, o^*)$. 
		\begin{equation*}
		J^* = \left(
		\begin{array}{ccc}
		-p_3 & - K_1 & - K_3 \\
		0 & -p_6 + K_2  & K_4 \\
		1 & 0 & -1 \\
		\end{array}
		\right) ,
		\end{equation*}
		where
		\begin{multline*}
		K_1 = \tfrac{A p_2 a^*}{(1 + p_2 a^* r^*)^2 } =
		\tfrac{2 A ( \sqrt{ 1 + \frac{4p_2 A}{p_3}\,r^*} - 1 ) }
		{r^*(1 +  \sqrt{ 1 + \frac{4p_2 A}{p_3}\,r^*}  )^2} =
		\tfrac{A p_2 \sqrt{p_4(p_6 r^* - p_5) (1 + p_5 - p_6 r^*)}}{r^*
			(p_2\sqrt{p_4(p_6 r^* - p_5)} + \sqrt{1 + p_5 - p_6 r^*}  )^2}  =  \\
		\tfrac{p_2 p_3 p_4 (p_6 r^* - p_5) }{(r^*)^2(p_2\sqrt{p_4(p_6 r^* - p_5)} + \sqrt{1 + p_5 - p_6 r^*}  )^2}
		\Bigl( 1 + p_2 \sqrt{\tfrac{p_4(p_6 r^* - p_5)}{1 + p_5 - p_6 r^*}}  \Bigr) \geqslant 0,
		\end{multline*}
		\begin{multline*}
		K_2 =  \tfrac{2 p_4 r^* (a^*)^2 }{(p_4 +  (a^* r^*)^2)^2} = \tfrac{1}{2 p_2^2 p_4 r^*}\bigl[ \sqrt{ 1 + \tfrac{4p_2 A}{p_3}\,r^*} - 1
		\bigr ]^2 (1 + p_5 - p_6 r^*)^2 =    \\
		\tfrac{2}{r^*}(p_6 r^* - p_5) (1 + p_5 - p_6 r^*) \geqslant 0 \text{ and } \ 0 \leqslant K_2 \leqslant 2 p_6 (  \sqrt{1+ p_5}  - \sqrt{ p_5} )^2,
		\end{multline*}
		\begin{eqnarray*}
			K_3 = \tfrac{r^*}{a^*} K_1 = \tfrac{  p_2 p_3 \sqrt{ p_4( p_6 r^* - p_5 ) }  } {  \sqrt{ 1 + p_5 - p_6 r^* }  + p_2 \sqrt{ p_4( p_6 r^* - p_5 ) }   } \text{ and }0 \leqslant K_3 \leqslant  p_3,
		\end{eqnarray*}
		
		$$
		K_4 =  \tfrac{r^*}{a^*} K_2 = \tfrac{ 2 r^*}{ \sqrt{p_4}}  (1 + p_5 - p_6 r^*)^{\frac{3}{2}} \sqrt{(p_6 r^* - p_5 )} \geqslant 0.
		$$
		Next, we will analyze the stability of the fixed point. First, we look for eigenvalues for $J^*$. So,
		$$
		|J^* - \lambda I | = \left |
		\begin{array}{ccc}
		-p_3 -\lambda & -K_1  &- K_3\\
		0 & -p_6 + K_2 - \lambda &   K_4\\
		1 & 0 & -1-\lambda \\
		\end{array}
		\right | =0,
		$$
		whence we obtain the characteristic equation:
		$$
		(\lambda +1)(\lambda + p_3)(\lambda +p_6 - K_2) +  K_3 (\lambda + p_6) = 0,
		$$
		i.\,e.
		\begin{equation}\label{eq-8}
		\lambda^3 + \alpha_1 \lambda^2 + \alpha_2 \lambda + \alpha_3 =0,
		\end{equation}
		where %$\alpha_i \geqslant 0$, $\alpha_1 \alpha_2 > \alpha_3$, and
		$$
		\alpha_1 = p_3 + p_6 - K_2 + 1,\
		\alpha_2 = %\alpha_1 + \alpha_3 - 1 + \frac{ r^*}{  a^* } K_1(1-p_6), \
		p_3 + p_6 - K_2 + p_3(p_6 - K_2) + K_3,\
		\alpha_3 = p_3(p_6 - K_2) + p_6 K_3.
		$$
		Let us denote by
		$$
		\Delta := 18 \alpha_1 \alpha_2 \alpha_3 - 4 \alpha_1^3 \alpha_3 + \alpha_1^2 \alpha_2^2 - 4 \alpha_2^3 -
		27 \alpha_3^2.
		$$
		If $\Delta > 0$, then (\ref{eq-8}) has three distinct real roots. If $\Delta = 0$, then (\ref{eq-8}) has
		a multiple root and all of its roots are real. If $\Delta < 0$, then (\ref{eq-8}) has
		one real root and two  complex  roots.
		
		To analyze stability, we will use Lemma \ref{H-0} in the Appendix. Let $ x_1:=\frac{p_5}{p_6}\leqslant x:= r^* \leqslant x_2:=\frac{p_5+1}{p_6}$. Then in our case,  
		
		\begin{equation*}\alpha_1>0 \Leftrightarrow 0 \leqslant K_2 < p_6 + p_3 + 1  \Leftrightarrow 
		2p_6^2  x^2  + (p_3 +1 - p_6(1+4p_5) ) x + 2p_5(p_5 +1) > 0
		\end{equation*}
		
		\noindent which is true for all $ p_i > 0$;
		
		\begin{equation*}
		\alpha_3>0 \Leftrightarrow 0 \leqslant K_2 < p_6  ( 1 + \tfrac{ K_3 }{p_3}  ) 
		\Leftrightarrow
		(x- x_1)(x - x_2) < \tfrac{ x }{2p_6} [ 1 + \tfrac{ p_2 \sqrt{ p_4(x - x_1) }  }{   \sqrt{ x_2 - x }  + p_2  \sqrt{ p_4(x - x_1) } }  ] 
		\end{equation*}
		which is true for all  $p_i > 0$;

		\begin{eqnarray*}\alpha_1\alpha_2>\alpha_3 \Leftrightarrow 0 \leqslant K_2  < p_6 + \tfrac{1}{2} \biggl[  p_3 + 1 + \tfrac{K_3}{ 1 + p_3} 
			\\- \sqrt{   (p_3 +2p_6 + 1 + \tfrac{K_3}{ 1 + p_3} )^2  - 4 [(p_3 +p_6)(1+p_6) + K_3 ]  }\biggr ] \\
			= p_6 + \tfrac{1}{2} [  p_3 + 1 + \tfrac{K_3}{ 1 + p_3}  - \sqrt{  ( \tfrac{K_3}{ 1 + p_3} )^2  - 2 [1 - \tfrac{2p_6}{1+p_3} ] K_3  + (p_3 -1)^2  } ]  \end{eqnarray*}   
		
		\noindent provided   
		$$
		\biggl( \tfrac{K_3}{ 1 + p_3}\biggr )^2  - 2 \biggl[1 - \tfrac{2p_6}{1+p_3} \biggr] K_3  + (p_3 -1)^2     > 0,    
		$$
		but if  $\biggl( \frac{K_3}{ 1 + p_3} \biggr)^2  - 2 \biggl[1 - \frac{2p_6}{1+p_3} \biggr] K_3  + (p_3 -1)^2   \leqslant 0$  then these are true for all $p_i>0$. Thus the fixed point is asymptotically stable.
	\end{proof}
	
	%%%%%%%%%%%%%%%%%%%%
	\subsection{Stability for different parameter values}
	
	\noindent\emph{Case 1:} If $A = 0$  then the system (\ref{eq-1})--(\ref{eq-3}) has the fixed point $\left(0,\frac{p_5}{p_6},0\right)$.
	The corresponding characteristic equation is
	$$
	(\lambda +1)(\lambda + p_3)(\lambda +p_6 ) = 0,
	$$
	whence $\lambda_{i} = -1,\,-p_3,\, -p_6 < 0$. As a result, $\left(0,\frac{p_5}{p_6},0\right)$ is stable node.

	\noindent\emph{Case 2:} If $p_2 = 0$ then the system (\ref{eq-1})--(\ref{eq-3}) has the fixed point $\left( \frac{A}{p_3}, r^*,\frac{A}{p_3}\right)$,
	where $r^*$ is a solution of the equation:
	$$
	\tfrac{p_4}{p_4 + (\tfrac{A}{p_3})^2 r^2} = 1 + p_5 - p_6 r.
	$$ 
	
	\noindent This equation can also be written as 
	\begin{eqnarray*} r\left[ \tfrac{p_6}{p_4p_5}\left(\tfrac{A}{p_3} \right)^2r^2-\tfrac{1+p_5}{p_4p_5}\left(\tfrac{A}{p_3}\right)^2r+\tfrac{p_6}{p_5}\right]=1\end{eqnarray*}
	which yields the following cases:
	\begin{itemize}
		\item if $p_4>\left[\frac{A(1+p_5)}{p_3p_6}\right]^2$ then there is one real root;
		\item if $p_4=\left[\frac{A(1+p_5)}{p_3p_6}\right]^2$ then $r\left( r-\frac{1+p_5}{2p_6}\right)^2=\frac{p_4p_5}{p_6}$, whence
		\begin{itemize}
			\item[$ $] if $\frac{p_5A^2}{p_3^2p_6}<\frac{1}{54}$ then we have three real roots,
			\item[$ $] if $\frac{p_5A^2}{p_3^2p_6}=\frac{1}{54}$ then we have two real roots,
			\item[$ $] if $\frac{p_5A^2}{p_3^2p_6}>\frac{1}{54}$ then we have one real root;
		\end{itemize}
		\item if $p_4<\left[\frac{A(1+p_5)}{p_3p_6}\right]^2$ then $r(r-r_1)(r-r_2)=\frac{p_4p_5}{p_6}\left(\frac{p_3}{A} \right)^2$,  where \\ 
		$r_{1,2}=\frac{1}{2p_6}\left[1+p_5\pm\sqrt{(p_5+1)^2-p_4p_6^2\left(\frac{p_3}{A}\right)^2} \right]$, whence 
		\begin{itemize}
			\item[$ $] if $ \frac{ p_4 p_5} {p_6 }  ( \frac{p_3}{A} )^2  <  \frac{ (r_1 + r_2 - K)(r_2 - 2 r_1 - K)(r_1 - 2 r_2 - K) }{27}$ then we have three real roots,
			\item[$ $] if $ \frac{ p_4 p_5} {p_6 }  ( \frac{p_3}{A} )^2  =  \frac{ (r_1 + r_2 - K)(r_2 - 2 r_1 - K)(r_1 - 2 r_2 - K) }{27}$ then we have two real roots,
			\item[$ $] if $ \frac{ p_4 p_5} {p_6 }  ( \frac{p_3}{A} )^2  >  \frac{ (r_1 + r_2 - K)(r_2 - 2 r_1 - K)(r_1 - 2 r_2 - K) }{27}$ then we have one real root, 
		\end{itemize}
		where $K^2 = r_1^2 - r_1 r_2 + r_2^2$.
	\end{itemize}
	
	\noindent The corresponding characteristic equation is
	$$
	(\lambda +1)(\lambda + p_3)(\lambda + p_6 - K_2 ) = 0,
	$$
	whence $\lambda_{i} = -1,\,-p_3,\, -p_6 + K_2$.
	%As $K_2 \leqslant 0$
	%then $( \frac{A}{p_3}, r^*,\frac{A}{p_3})$ is stable node.
	%
	If $K_2 < p_6$
	then $( \frac{A}{p_3}, r^*,\frac{A}{p_3})$ is stable node. If $K_2 > p_6$
	then $( \frac{A}{p_3}, r^*,\frac{A}{p_3})$ is saddle. If $K_2 = p_6$
	then it is a non-hyperbolic fixed point.

	\noindent\emph{Case 3:} If $p_3 = 0$  then the system (\ref{eq-1})--(\ref{eq-3}) has the fixed point $(+\infty,\frac{p_5 +1}{p_6},+\infty)$.
	The corresponding characteristic equation is
	$$
	\lambda(\lambda +1)(\lambda +p_6 ) = 0,
	$$
	whence $\lambda_{i} = -1,\,0,\, -p_6$. As a result, $(+\infty,\frac{p_5 +1}{p_6},+\infty)$ is non-hyperbolic fixed point.

	\noindent\emph{Case 4:} If $p_4 = 0$ then the system (\ref{eq-1})--(\ref{eq-3}) has the fixed point $\left(a^*,\frac{p_5 +1}{p_6},a^*\right)$,
	where  $$a^* =  \frac{p_6}{2p_2(p_5 +1)} \left[ \sqrt{1+ \frac{4 A p_2(p_5+1)}{p_3 p_6}} -1 \right]. $$
	In this case,  we have that $K_2 =K_3=0$ and $(\lambda +1)(\lambda + p_3)(\lambda + p_6) =0,$ whence $\lambda_i = -1,\, - p_3, \, -p_6$. Hence, the fixed point is a stable node.
	
	\noindent\emph{Case 5:}
	If  $p_2 = p_4 = 0$ then we obtain the explicit solution
	
	\begin{eqnarray*}
		a(t) &=& (a_0 - \tfrac{A}{p_3} ) e^{-p_3 t}  + \tfrac{A}{p_3}  \to  \tfrac{A}{p_3} \text{ as }t \to +\infty,\\   
		r(t) &=&  ( r_0 -  \tfrac{ p_5 +1}{p_6} )e^{-p_6 t}  + \tfrac{1+p_5}{p_6}  \to  \tfrac{1+p_5}{p_6} \text{ as }t \to +\infty,\\
		o(t) &=& (o_0 - \tfrac{A}{p_3} ) e^{-t}  + (a_0 - \tfrac{A}{p_3} ) e^{-t} \int \limits_0^t {  e^{(1-p_3) s} ds}  +  \tfrac{A}{p_3}  \to  \tfrac{A}{p_3} \text{ as }t \to +\infty. 
	\end{eqnarray*}

	\noindent\emph{Case 6:} If $p_5 = 0$ then  the one of fixed points  is $\left(\frac{A}{p_3}, 0, \frac{A}{p_3}\right)$  and
	
	$$(\lambda +1)(\lambda + p_3) (\lambda + p_6) =0,$$
	
	\noindent whence $\lambda_i = -1,\, - p_3, \, -p_6$. Hence, this fixed point is stable node. In this case, by (2.5) we obtain that

	$$\tfrac{A}{1 + p_2 a r } = p_3 a   \Leftrightarrow   r = \tfrac{1}{ p_2 a } ( \tfrac{A}{p_3 a}  - 1) \quad \text{provided} \quad 0 < a < \tfrac{A}{p_3}, $$     
	whence
	
	$$\tfrac{ (a r)^2 } { p_4 + (ar)^2  } = p_6 r$$
	implies that   $$r = 0
	\text{ or }   \tfrac{ a^2 r  } { p_4 + (ar)^2  } = p_6.$$  
	
	\noindent Hence, we find that
	
	$$a^3  + (   \tfrac{p_6(1+ p_2^2 p_4 ) }{p_2} -  \tfrac{A}{p_3} ) a^2  - 2 \tfrac{A p_6} {p_2 p_3} a = - \tfrac{p_6}{p_2} ( \tfrac{A}{p_3} )^2 \Leftrightarrow f(a) := a (a - a_1)(a - a_2) = - \tfrac{p_6}{p_2} ( \tfrac{A}{p_3} )^2,$$
	
	\noindent where
	
	$$a_{1,2} = \tfrac{1}{2} \biggl[  - (   \tfrac{p_6(1+ p_2^2 p_4 ) }{p_2} -  \tfrac{A}{p_3} )  \pm  \sqrt{   (   \tfrac{p_6(1+ p_2^2 p_4 ) }{p_2} -  \tfrac{A}{p_3} )^2 +  8 \tfrac{A p_6} {p_2 p_3}  }  \biggr] $$
	
	\noindent and $  a_1 < 0 < a_2  $.   Note that  $ a_2 \leqslant \frac{A}{p_3}$ provided  $p_2^2 p_4 \geqslant 1$. As a result, 
	
	\begin{itemize}
		\item  if  $p_2^2 p_4 \geqslant 1$ then
		\begin{itemize}
			\item[$ $] if  $f_{min}  >   - \frac{p_6}{p_2} ( \frac{A}{p_3} )^2$  then no real roots;
			
			\item[$ $] if  $f_{min}  =   - \frac{p_6}{p_2} ( \frac{A}{p_3} )^2$  then one positive real root;
			
			\item[$ $] if  $f_{min}  <   - \frac{p_6}{p_2} ( \frac{A}{p_3} )^2$  then two positive real roots;
		\end{itemize}
		\item  if  $p_2^2 p_4 < 1$ then
		\begin{itemize}
			\item[$ $] if  $f_{min}  \geqslant   - \frac{p_6}{p_2} ( \frac{A}{p_3} )^2$  then no real roots;
			
			\item[$ $] if  $f_{min}  <   - \frac{p_6}{p_2} ( \frac{A}{p_3} )^2$  then one positive real root.
		\end{itemize}
	\end{itemize}
	
	\noindent\emph{Case 7:}
	If  $p_6 = p_5 = 0$  then we have the following system
	\begin{equation*}
	a'(t) = \tfrac{A}{1 + p_2 or}  - p_3 a,\quad
	r'(t) = \tfrac{(or)^2}{p_4 + (or)^2 } , \quad
	o'(t) = a - o.
	\end{equation*}
	
	\noindent If  $r_0 = 0$ then we find the explicit solution
	
	\begin{equation*}
	a(t) = (a_0  - \tfrac{A}{p_3} ) e^{- p_3 t} + \tfrac{A}{p_3}, \quad
	r(t) = 0,\quad
	o(t) =  (o_0  - \tfrac{A}{p_3} ) e^{- t} + (o_0  - \tfrac{A}{p_3} ) e^{- t}  \int_0^t{ e^{(1- p_3)s } ds}   + \tfrac{A}{p_3}.
	\end{equation*}
	
	\noindent If  $r_0 \neq 0$ then we approximately have
	
	\begin{equation*}
	a'(t)  \approx  -p_3\left( a - \tfrac{A}{p_3} \right) - \tfrac{A^2 p_2}{p_3} r ,\quad
	r'(t)  \approx   \tfrac{1}{p_4} \left(  \tfrac{A}{p_3} \right)^2 r^2,\quad
	o'(t) = a - o,
	\end{equation*}
	
	\noindent whence
	
	\begin{eqnarray*}
		r(t) &\approx& \tfrac{r_0}{ 1  -  \frac{r_0}{p_4} (  \frac{A}{p_3} )^2  t }  \to +\infty \text{ as }  t \to T^* := \tfrac{p_4}{r_0} \left(  \tfrac{p_3}{A} \right)^2,\\
		a(t) &\approx& a_0 e^{-p_3 t} + \tfrac{ A}{p_3} (1- e^{-p_3 t} )  -  \tfrac{A^2 p_2}{p_3} e^{-p_3 t} \int_0^t { r(s) e^{p_3 s} ds },\\
		o(t)  &\approx& o_0 e^{- t} + e^{- t} \int_0^t { a(s) e^{ s} ds }.
	\end{eqnarray*}
	
	\noindent If  $p_6 =0$ but  $ p_5 \neq 0$  then  $r(t)$ blows up in a finite time too.
	
	Also, note that if  $ p_6 = 0$ and $p_5 =  0$ then the system (\ref{eq-1})--(\ref{eq-3})
	has the fixed point $( \frac{A}{p_3},0, \frac{A}{p_3})$. The corresponding characteristic equation is
	$$
	\lambda (\lambda +1)(\lambda + p_3)  = 0,
	$$
	whence $\lambda_{i} = -1,\,-p_3,\, 0$. As a result, $( \frac{A}{p_3},0, \frac{A}{p_3})$ is non-hyperbolic fixed point.

	\begin{example}\label{ex:stable}
		Let $A=1$, $p_2 = 15$, $p_3 = 7.2$, $p_4 = 0.05$, $p_5 = 0.11$, and $p_6 = 2.9$.
		Then $r^* \approx 0.03$, $a^* = o^* \approx 0.12$,
		$\alpha_1 = 11.1 - K_2  \approx 11.07$,
		$\alpha_2 = 30.98 - 8.2 K_2 +  K_3  \approx 30.78$,
		$\alpha_3 = 20.88 - 7.2 K_2 + 2.9  K_3   \approx 20.75$,
		whence we find that $\Delta \approx 2509.05 > 0$, $\alpha_1 \alpha_2 > \alpha_3$. As a result, all characteristic roots are negative
		real numbers and the fixed point is stable node.
		
		A visual representation of this stable node can be found in Figure \ref{fig:stable}. This plot was created using the Matlab ode45 solver (\cite{Matlab}) using various starting values and the parameter values given above. The starting values were selected so that $a_0>0$, $r_0>0$ and $o_0>0$ to imitate real initial hormone levels.
		
		\begin{figure}
			\centering
			\includegraphics[width=.75\textwidth]{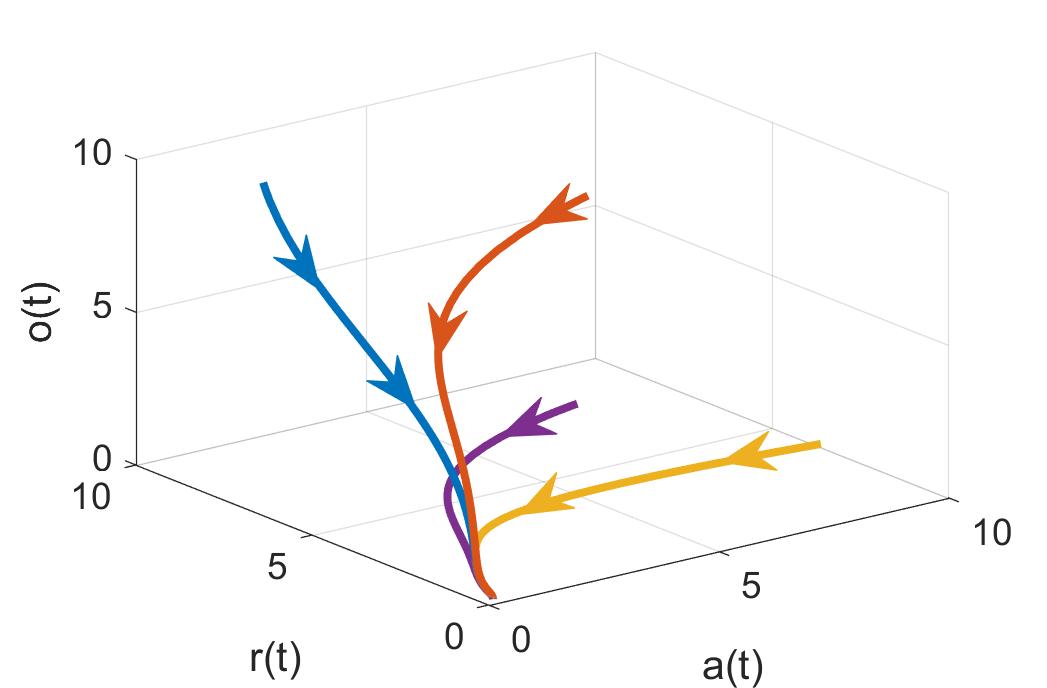}
			\caption{A plot of different trajectories illustrating the stable node associated with parameter values given in Example \ref{ex:stable}.}
			\label{fig:stable}
		\end{figure}
	\end{example}
	
	\begin{example}\label{ex:unstable}
		Let $A=0.106$, $p_2 = 0$, $p_3 = 0.222$, $p_4 = 0.464$, $p_5 = 0.094$, and $p_6 = 0.418$.
		Then $r^* \approx 0.39,0.83,1.38$ and $a^* = o^* \approx 0.47$. Using similar calculations as above according to the defined values. If $r^*\approx 0.39$ then $\alpha_1\approx1.30$, $\alpha_2\approx0.32$, $\alpha_3\approx0.01$, and $K_2\approx0.33<p_6$ which means this is a stable node. If $r^*\approx 0.83$ then $\alpha_1\approx1.18$, $\alpha_2 \approx 0.17$, and $\alpha_3\approx-0.008$, and $K_2\approx0.45>p_6$ which means this is a saddle. If $r^*\approx 1.38$ then $\alpha_1\approx1.28$, $\alpha_2\approx0.29$, $\alpha_3\approx0.01$, and $K_2\approx0.35<p_6$ which means this is a stable node.

		This is illustrated in Figure \ref{fig:unstable} using the stated above parameter values. The starting values were selected so that $a_0>0$, $r_0>0$ and $o_0>0$ to imitate real initial hormone levels.
		
		\begin{figure}
			\centering
			\includegraphics[width=.75\textwidth]{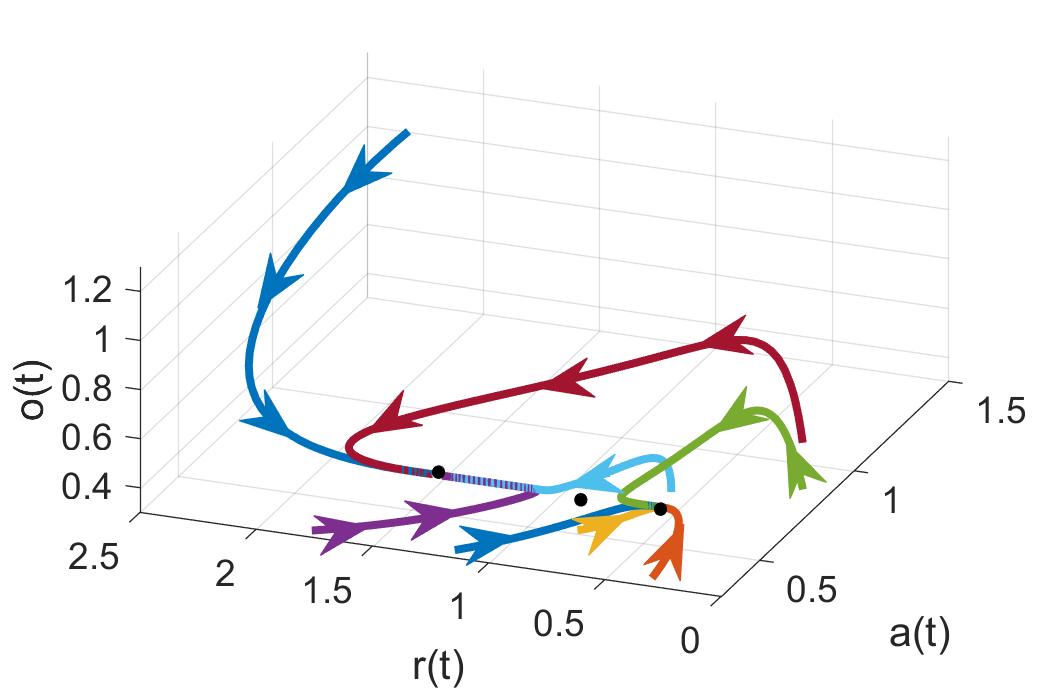}
			\caption{A plot of different trajectories illustrating the unstable saddle-node with only realistic initial conditions and the above parameter values stated in Example \ref{ex:unstable}.}
			\label{fig:unstable}
		\end{figure}
	\end{example}
	
	%%%%%%%%%%%%%%%%%%%%%%%%%%%%%%
	
	\subsection{Lyapunov stability analysis}
	In this section, we show the stability of the fixed point by using the Lyapunov function approach. We consider the system (\ref{eq-1})-(\ref{eq-3}) and denote 
	$$
	W(t) := \tfrac{1}{2} [ (a(t)-a^*)^2 + (r(t)-r^*)^2 + (o(t)-o^*)^2 + (o(t)r(t)-o^*r^* )^2],
	$$
	where $(a^*,r^*,o^*)$ is the fixed point ($a^*= o^*$).
	
	\begin{lemma}[Stability]\label{stab}
		Assume that
		$$
		A \geqslant 0,  \ p_2 \geqslant 0, \ p_3 > \tfrac{1}{2},\ p_4 \geqslant 0,  \
		p_6 >  \tfrac{1}{\min \{ p_3 - \tfrac{1}{2},  p_6, 1 \}},
		$$
		and
		$$
		0 \leqslant  p_5 <  p_6 \min \{ p_3 - \tfrac{1}{2},  p_6, 1 \} - 1.
		$$
		Then there exist $W^* > 0$, $A_0 > 0$ and $p^*_4 > 0$ such that
		\begin{equation}\label{rrr-t}
		W(t) \to 0 \text{ as } t \to +\infty
		\end{equation}
		provided
		$
		W(0) < W^*,\ 0 \leqslant A < A_0, \ p_4 > p_4^*,
		$
		hence, the fixed point $(a^*,r^*,o^*)$ is globally stable. If $ p_4 \leqslant p_4^* $ then
		there exist  $A_0 \leqslant A_1 < A_2 $ such that (\ref{rrr-t}) holds provided
		$
		W(0) < W^{*},\   A_1 < A < A_2.
		$
	\end{lemma}

	\begin{proof}[Proof of Lemma \ref{stab}]
		
		Using  the system (\ref{eq-1})--(\ref{eq-3}), we have
		\begin{multline*}
		\tfrac{d}{dt}W(t) -  (o r - o^*r^*) (o(t)r(t))' = 
		-  p_3 (a-a^*)^{2 } -   p_6 (r-r^*)^{2} -  (o-o^*)^{2 } + \\
		(a-a^*)  \bigl[ \tfrac{A}{1+p_2 or} -   \tfrac{A}{1+p_2 o^* r^* } \bigr]
		+(r-r^*)  \bigl[ \tfrac{p_4}{p_4 + ( o^*r^*)^2 } - \tfrac{p_4}{p_4 + ( or)^2 } \bigr ] + 
		(a - o^*) (o-o^*) .
		\end{multline*}
		As
		$$
		2 (a - o^*) (o-o^*)  \leqslant (a - a^*)^2 +  (o-o^*)^{2}, \quad
		\bigl | \tfrac{1}{1+p_2 or} -   \tfrac{1}{1+p_2 o^* r^* } \bigr | \leqslant p_2 | or - o^* r^* |,
		$$
		\begin{multline*}
		(o(t)r(t))' = r (a - o) + o [ - \tfrac{p_4}{p_4 + ( or)^2 } + 1 + p_5 - p_6 r]  =  
		(r - r^*)(a - a^*) + r^*(a-a^*) +\\ a^* (r - r^*) - (o r - o^* r^*) - 
		p_6 (r - r^*)(o - o^*) - p_6 o^*(r - r^*) +  (o - o^*) \bigl[ \tfrac{p_4}{p_4 + ( o^*r^*)^2 } - \tfrac{p_4}{p_4 + ( or)^2 } \bigr ]
		+ \\
		o^* \bigl[ \tfrac{p_4}{p_4 + ( o^*r^*)^2 } - \tfrac{p_4}{p_4 + ( or)^2 } \bigr ],
		\end{multline*}
		
		$$
		\bigl | \tfrac{p_4}{p_4 + ( o^*r^*)^2 } - \tfrac{p_4}{ p_4 + ( or)^2 } \bigr | \leqslant
		\tfrac{ | or - o^* r^* | \cdot | or + o^* r^* |}{ p_4 + ( o^*r^*)^2 }, \quad
		| or + o^* r^* | \leqslant | or - o^* r^* | + 2 o^* r^*,
		$$
		then
		\begin{equation*}
		\tfrac{d}{dt}W(t) \leqslant - \alpha  W(t) + \beta W^{ \frac{3}{2}}(t) + \gamma   W^{2}(t),
		\end{equation*}
		i.\,e.
		\begin{equation}\label{rrr-1}
		\tfrac{d}{dt}W(t) \leqslant \gamma W(t)
		\biggl[  W^{ \frac{1}{2}}(t) + \tfrac{ \beta - \sqrt{\beta^2 + 4 \alpha \gamma}}{2\gamma} \biggr]
		\biggl[  W^{ \frac{1}{2}}(t) + \tfrac{ \beta + \sqrt{\beta^2 + 4 \alpha \gamma}}{2\gamma} \biggr],
		\end{equation}
		where
		$$
		\alpha = 2 [ \min \{ p_3 - \tfrac{1}{2},  p_6, 1 \} - A p_2 - r^* - (p_6 + 1) a^* -
		\tfrac{ 4 o^* r^* }{ p_4 + ( o^*r^*)^2 } ]  > 0
		$$
		$$
		\beta = 2^{\frac{3}{2}} [ p_6 + 1  + \tfrac{ 3  o^* r^* }{ p_4 + ( o^*r^*)^2 }]
		\leqslant  2^{\frac{3}{2}} [ p_6 + 1  + 3   \min \bigl \{ \tfrac{1}{2 p_4^{\frac{1}{2}}},
		\tfrac{2 p_2 }{ \sqrt{ 1 + \frac{4 p_2 p_5 A}{p_6} } - 1 } \bigr \} ],
		$$
		$$
		\gamma = \tfrac{ 4 o^* r^* }{ p_4 + ( o^*r^*)^2 } \leqslant 4 \min \bigl \{ \tfrac{1}{2 p_4^{\frac{1}{2}}},
		\tfrac{2 p_2 }{ \sqrt{ 1 + \frac{4 p_2 p_5 A}{p_6} } - 1 } \bigr \}  ,
		$$
		provided
		\begin{equation}\label{rrr-0}
		0 < r^* + (p_6 + 1) a^* + \tfrac{ 4 o^* r^* }{ p_4 + ( o^*r^*)^2 } < \min \{ p_3 - \tfrac{1}{2},  p_6, 1 \} - A p_2 .
		\end{equation}
		As $0 \leqslant a^* = o^* \leqslant \frac{A}{p_3}$, $\frac{p_5}{p_6} \leqslant r^* \leqslant \frac{p_5 +1}{p_6} $
		and $o^* r^*  \geqslant \frac{1}{2p_2}  [ \sqrt{ 1 + \frac{4 p_2 p_5 A}{p_6} } - 1  ]$
		then by (\ref{rrr-0}) we get
		$$
		\tfrac{p_6 + 1 + p_2 p_3}{p_3} A + \min \bigl \{ \tfrac{2}{p_4^{\frac{1}{2}}},
		\tfrac{8 p_2 }{ \sqrt{ 1 + \frac{4 p_2 p_5 A}{p_6} } - 1 } \bigr \} < B:= \min \{ p_3 - \tfrac{1}{2},  p_6, 1 \} -
		\tfrac{p_5 + 1}{p_6}.
		$$
		Hence,
		$$
		\tfrac{p_6 + 1 + p_2 p_3}{p_3} A + \tfrac{2}{p_4^{\frac{1}{2}}} < B \text{ and } A \leqslant \tfrac{2 p_6 p_4^{\frac{1}{2}} }{p_5} (1+ 2 p_2 p_4^{\frac{1}{2}} ),
		$$
		whence
		$$
		0 \leqslant A < A_0 := \min \{ \tfrac{2 p_6 p_4^{\frac{1}{2}} }{p_5} (1+ 2 p_2 p_4^{\frac{1}{2}} ),
		\tfrac{p_3}{p_6 + 1 + p_2 p_3}( B - \tfrac{2}{p_4^{\frac{1}{2}}} ) \},
		$$
		or
		$$
		F(A) := \tfrac{p_6 + 1 + p_2 p_3}{p_3} A + \tfrac{8 p_2 }{ \sqrt{ 1 + \frac{4 p_2 p_5 A}{p_6} } - 1 } < B \text{ and } A > \tfrac{2 p_6 p_4^{\frac{1}{2}} }{p_5} (1+ 2 p_2 p_4^{\frac{1}{2}} ).
		$$
		As the function $F(A)$ has a unique minimum for positive $A$, denote by $A_{\min}$,
		then there exist $0< A_1 < A_{\min} < A_2 $ such that $F(A) < B$ provided
		$F(A_{\min}) < B$.

		So, if $W(0) <  [\frac{ \sqrt{\beta^2 + 4 \alpha \gamma}- \beta}{2\gamma} ]^2$ then by (\ref{rrr-1}) we deduce that
		$$
		W(t) \to 0 \text{ as } t \to +\infty.
		$$
	\end{proof}
	
	%%%%%%%%%%%%%%%%%%%%%%%%%%%%%%%%%%%%%%%%%%%%%%%%%%
	
	\section{Analysis of the model with time
		delay}\label{sec-3}
	
	\subsection{Stability analysis with respect to time delay}
	
	Note the fixed point $(a^*,r^*,o^*)$ for (\ref{eq-1})--(\ref{eq-3}) coincides with the one for  (\ref{a-1})--(\ref{a-3}).
	Let us denote by
	$$
	J_{\tau} := \left(
	\begin{array}{ccc}
	\frac{\partial f_1}{\partial a_{\tau}} &  \frac{\partial f_1}{\partial r_{\tau}} &  \frac{\partial f_1}{\partial o_{\tau}} \\
	\frac{\partial f_2}{\partial a_{\tau}} &  \frac{\partial f_2}{\partial r_{\tau}} &  \frac{\partial f_2}{\partial o_{\tau}} \\
	\frac{\partial f_3}{\partial a_{\tau}} &  \frac{\partial f_3}{\partial r_{\tau}} &  \frac{\partial f_3}{\partial o_{\tau}}\\
	\end{array}
	\right),
	$$
	where $a_{\tau} = a(t-\tau)$, $r_{\tau} = r(t-\tau)$, and $o_{\tau} = o(t-\tau)$.
	Then $J_{\tau}$ at the point $(a^*,r^*,a^*)$ is equal
	$$
	J^*_{\tau} := \left(
	\begin{array}{ccc}
	0 &  0 &  0 \\
	0 &  0 &  0 \\
	1 &  0 &  0\\
	\end{array}
	\right).
	$$
	Now we will look for eigenvalues for the matrix $J^* + e^{-\lambda \tau} J^*_{\tau}$. So,
	$$
	|J^* + e^{-\lambda \tau} J^*_{\tau} - \lambda I | = \left |
	\begin{array}{ccc}
	-p_3 -\lambda & -K_1  &-  K_3\\
	0 & -p_6 + K_2 - \lambda &   K_4\\
	1+e^{-\lambda \tau}  & 0 & -1-\lambda \\
	\end{array}
	\right | =0,
	$$
	whence we obtain the characteristic equation: 
	$$(\lambda+1)(\lambda+p_3)(\lambda+p_6-K_2)=-(1+e^{-\lambda\tau})K_3(\lambda+p_6)$$
	
	Time delays are known to affect the stability of a fixed point. They can induce stability switches in which the zeros of the characteristic equation may cross the imaginary axis as the delay, $\tau$, increases. Looking at the characteristic equation as a function of $\tau$, and analyzing the location of the roots and the direction of motion as they cross the imaginary axis (see \cite{Cook}). Destabilization will happen at critical values $\tau_c$ which is when there is a pair of purely imaginary characteristic values. Following the ideas of papers \cite{Cook} and \cite{Barresi}, let's rewrite the characteristic equation as 
	\begin{eqnarray}\label{eq:98}
	C(\lambda)&:=&(\lambda+1)(\lambda+p_3)(\lambda+p_6-K_2)+(1+e^{-\lambda \tau})K_3(\lambda+p_6)\nonumber\\ 
	&=&((\lambda+1)(\lambda+p_3)(\lambda+p_6-K_2)+K_3(\lambda+p_6))+e^{-\lambda \tau}K_3(\lambda+p_6)\nonumber\\
	&=&P(\lambda)+Q(\lambda)e^{-\lambda\tau}=0
	\end{eqnarray}  
	Then define
	\begin{eqnarray}\label{eq:101}
	F(y)&=&|P(iy)|^2-|Q(iy)|^2 \nonumber\\
	&=&y^6+(p_6^2-2K_2p_6+p_3^2-2K_3+K_2^2+1)y^4+(K_2^2-2K_2K_3+2K_3p_3\nonumber\\
	&&-2K_2K_3p_3+p_3^2+K_2^2 p_3^2-2K_2p_6+2K_2K_3p_6-2K_2p_3^2p_6+p_6^2-2K_3p_6^2\nonumber\\
	&&+p_3^2p_6^2)y^2+(K_2^2p_3^2-2K_2K_3p_3p_6-2K_2p_3^2p_6+2K_3p_3p_6^2+p_3^2p_6^2).
	\end{eqnarray}
	We want to use Theorem \ref{delay} from the Appendix, so we need to check the following conditions:
	\begin{enumerate}
		\item $P(\lambda)=(\lambda+1)(\lambda+p_3)(\lambda+p_6-K_2)+K_3(\lambda+p_6)$ and $Q(\lambda)=K_3(\lambda+p_6)$ have no common imaginary zeros since each $p_i$ are real values. 
		\item It is quick to see that $\overline{P(i\lambda)}=P(i\lambda)$ and $\overline{Q(i\lambda)}=Q(i\lambda)$ for real $\lambda$.
		\item $P(0)+Q(0)=p_3(p_6-K_2) + 2 K_3 p_6 \neq0$ so this is an important restriction in order to use the Theorem \ref{delay}.
		\item Referring back to (\ref{eq-8}), we see that there are at most 3 roots of (\ref{eq:98}) if $\tau=0$.
		\item (\ref{eq:101}) has at most 6 real zeros for real $y$.
	\end{enumerate}
	Therefore, by Theorem \ref{delay} from the Appendix, if $F(y)$ has no positive roots, the system is stable for all $\tau\geqslant0$. If $F(y)$ has a simple positive root $y_0$, then there exists a pair of purely imaginary roots $\pm iv_0$ such that $v_0=\sqrt{y_0}$. For this $v_0$, there is a countable sequence of $\{\tau_0^n\}$ of delays for which stability switches can occur. Also, there exists a positive $\tau_c$ such that the system is unstable for all $\tau>\tau_c$. Investigating this further, let $x=y^2$
	\begin{eqnarray}\label{eq:99}
	F(x)&=&x^3+b_1x^2+b_2x+b_3,
	\end{eqnarray}
	where $b_1=p_6^2-2K_2p_6+p_3^2-2K_3+K_2^2+1$, $b_2=K_2^2+-2K_2K_3+2K_3p_3-2K_2K_3p_3+p_3^2+K_2^2p_3^2-2K_2p_6+2K_2K_3p_6-2K_2p_3^2p_6+p_6^2-2K_3p_6^2+p_3^2p_6^2$, and $b_3=K_2^2p_3^2-2K_2K_3p_3p_6-2K_2p_3^2p_6+2K_3p_3p_6^2+p_3^2p_6^2$. Note that 
	$$F'(x)=3x^2+2b_1x+b_2$$
	and
	\begin{equation}
	\label{d0}
	\Delta_0=b_1^2-3b_2.
	\end{equation}
	
	Now analyzing the roots of (\ref{eq:99}),
	\begin{itemize}
		\item If $\Delta_0\leqslant0$, then $F'(0)\geqslant0$ and $F(x)$ is monotonically non-decreasing. Further,
		\begin{itemize}
			\item if $F(0)>0$, then $F$ has no positive roots and all the roots of the characteristic will remain to the left of the imaginary axis for all $\tau>0$.
			\item if $F(0)<0$, then since $\lim_{x\rightarrow\infty} F(x)=\infty$, there is at least one positive root of $F$ and thus the roots of the characteristic equation can cross the imaginary axis.
		\end{itemize}
		\item If $\Delta_0>0$ then $F$ has critical points
		\begin{eqnarray*}
			x_{c_1}=\frac{-b_1+\sqrt{\Delta_0}}{3}&& x_{c_2}=\frac{-b_1-\sqrt{\Delta_0}}{3}
		\end{eqnarray*}
		and if $x_{c_1}>0$ and $F(x_{c_1})<0$, then $F$ has positive roots (see  \cite{Cook}).
	\end{itemize}
	
	Stability switches are possible for each positive simple root $x_j$ of (\ref{eq:99}) and the cross is from left to right if $F'(v_0)>0$, and from right to left if $F'(v_0)<0$ according to Theorem 1 (see \cite{Cook}). Now let's analyze the characteristic quasi-polynomial (\ref{eq:98}) for $\lambda=iv$:
	\begin{eqnarray}\label{eq:100}
	C(iv)=A_1-A_2\cos(v\tau)-A_3\sin(v\tau)+ i[A_4
	-A_3\cos(v\tau)+A_2\sin(v\tau)]=0, \nonumber
	\end{eqnarray}
	where 
	\begin{eqnarray*}
		A_1(v)=p_3p_6-K_2p_3+K_3p_6-v^2(p_6+p_3-K_2+1), && A_2=-K_3p_6.\\
		A_4(v)=v(p_3-K_2-K_2p_3+p_6+p_3p_6+K_3)-v^3, && A_3(v)=-K_3v.
	\end{eqnarray*}
	So $x_j$ ($j=1,2,3$) is a positive root of $F(x)=0$ and $v_j=\sqrt{x_j}$. Then $v_j$ satisfies (\ref{eq:100}) if its a solution to the system 
	\begin{eqnarray*}
		\left\{\begin{array}{l}A_1(v)-A_2\cos(v\tau)-A_3(v)\sin(v\tau)=0,\\A_4(v)-A_3(v)\cos(v\tau)+A_2\sin(v\tau)=0. \end{array} \right.
	\end{eqnarray*}
	This yields
	\begin{eqnarray*}
		\sin(v \tau)=\tfrac{A_1(v)A_3(v)-A_2A_4(v)}{A_2^2+A_3^2(v)},&& 
		\cos(v\tau)=\tfrac{ A_1(v) A_2 + A_3(v) A_4(v)}{ A_2^2 + A_3^2(v) },
	\end{eqnarray*}
	provided  $\max \{  |A_1(v) A_3(v) - A_2 A_4(v)|,  |A_1(v) A_2 - A_3(v) A_4(v)|  \}  \leqslant   A_2^2 + A_3^2(v)$,

	Therefore, for every positive root $v_j$, it yields the following sequence of delays $\{\tau_j^n\}$ for which there are pure imaginary roots (\ref{eq:98}):
	\begin{eqnarray}\label{eq:103}
	\tau_j^n=\tfrac{1}{v_j}\left\{\arctan\left(\tfrac{A_1(v_j)A_3(v_j)-A_2A_4(v_j)}{A_1(v_j)A_2+A_3(v_j)A_4(v_j)} + \pi \, n \right) \right\} && \text{for } n=0,1,2, \ldots
	\end{eqnarray}

	\begin{figure}
		\centering
		\begin{subfigure}[b]{0.45\textwidth}
			\includegraphics[width=\textwidth]{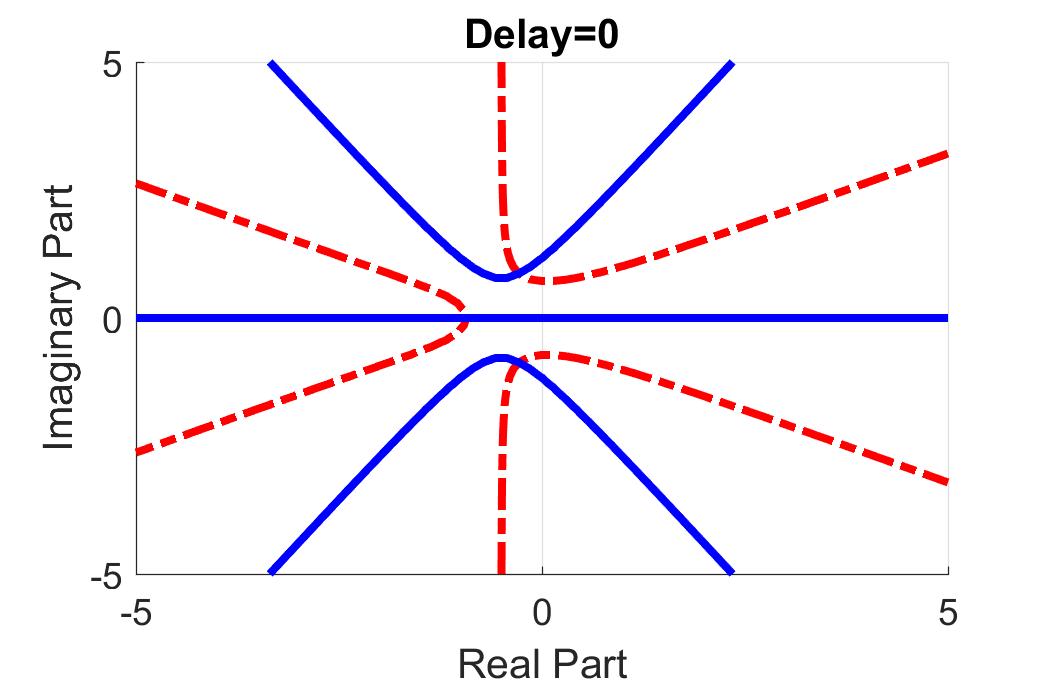}
		\end{subfigure}
		\quad
		\begin{subfigure}[b]{0.45\textwidth}
			\includegraphics[width=\textwidth]{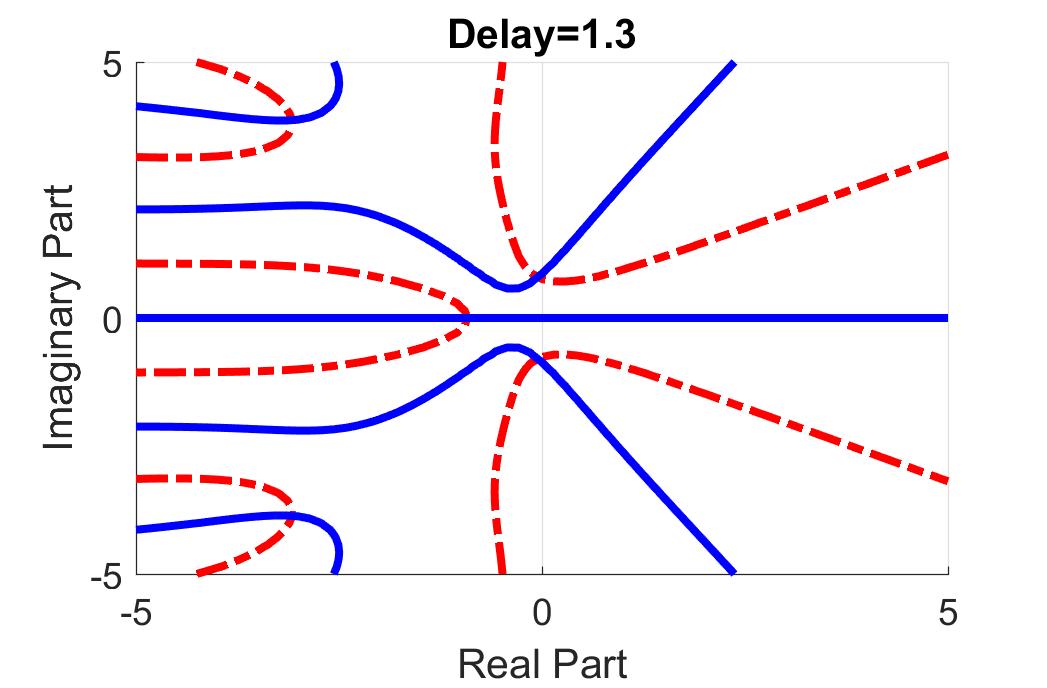}
		\end{subfigure}
		\vskip\baselineskip
		\begin{subfigure}[b]{0.45\textwidth}
			\includegraphics[width=\textwidth]{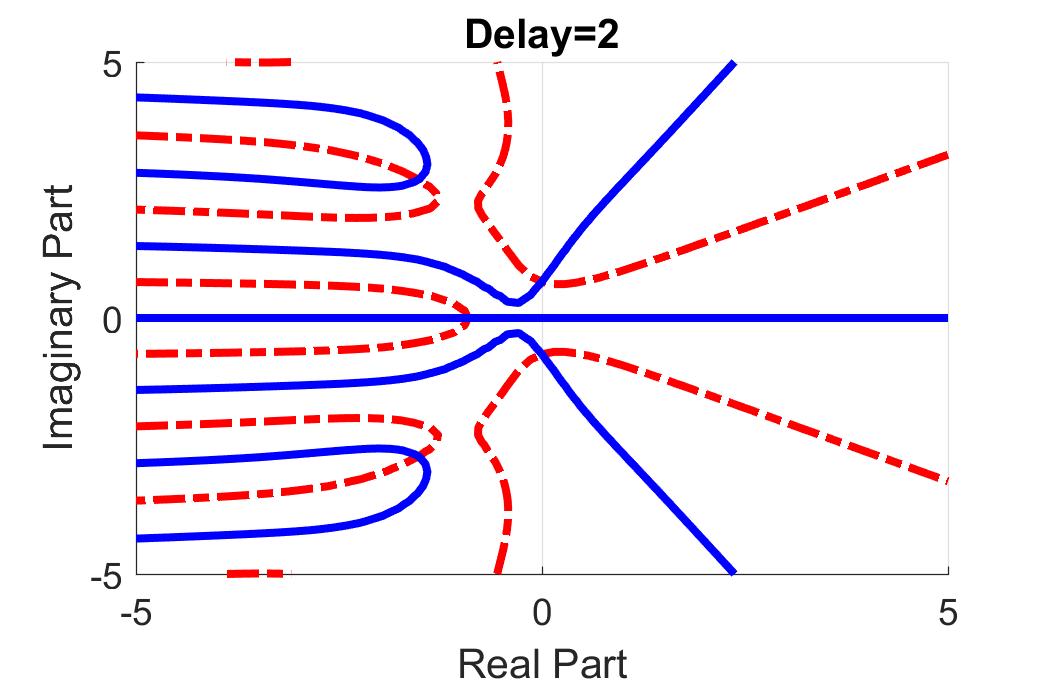}
		\end{subfigure}
		\quad
		\begin{subfigure}[b]{0.45\textwidth}
			\includegraphics[width=\textwidth]{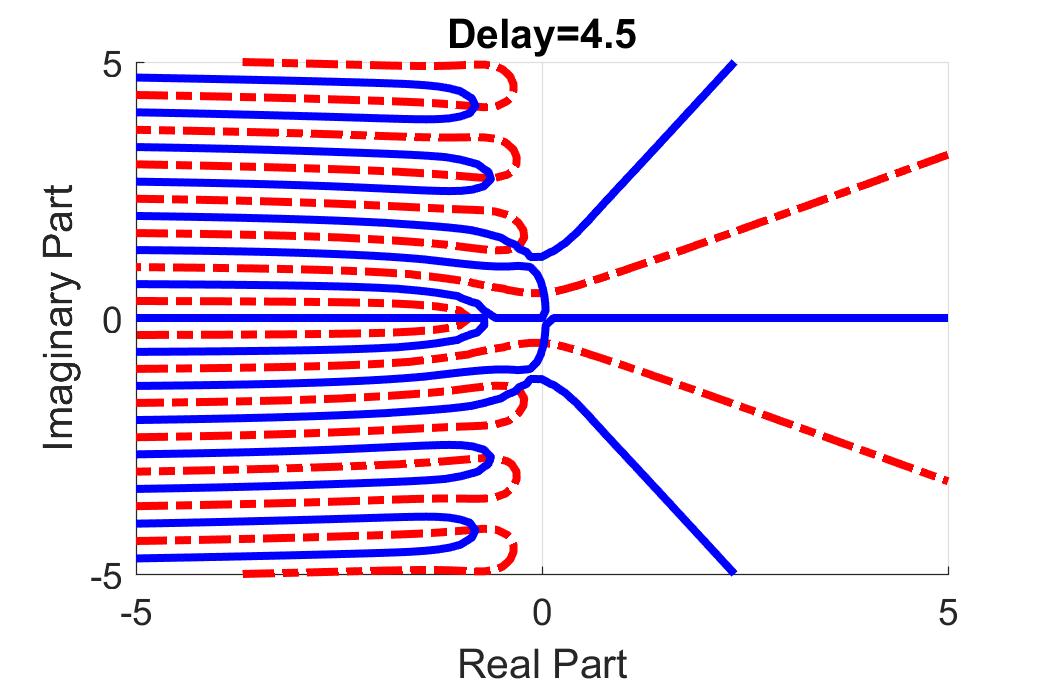}
		\end{subfigure}
		\vskip\baselineskip
		\begin{subfigure}[b]{0.45\textwidth}
			\includegraphics[width=\textwidth]{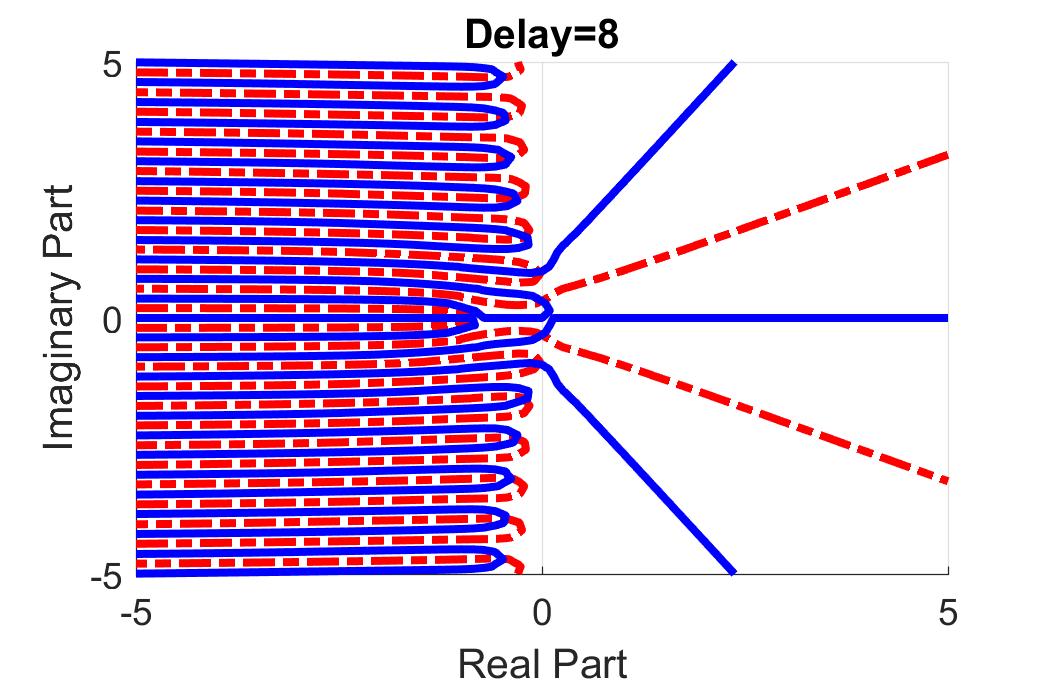}
		\end{subfigure}
		\quad
		\begin{subfigure}[b]{0.45\textwidth}
			\includegraphics[width=\textwidth]{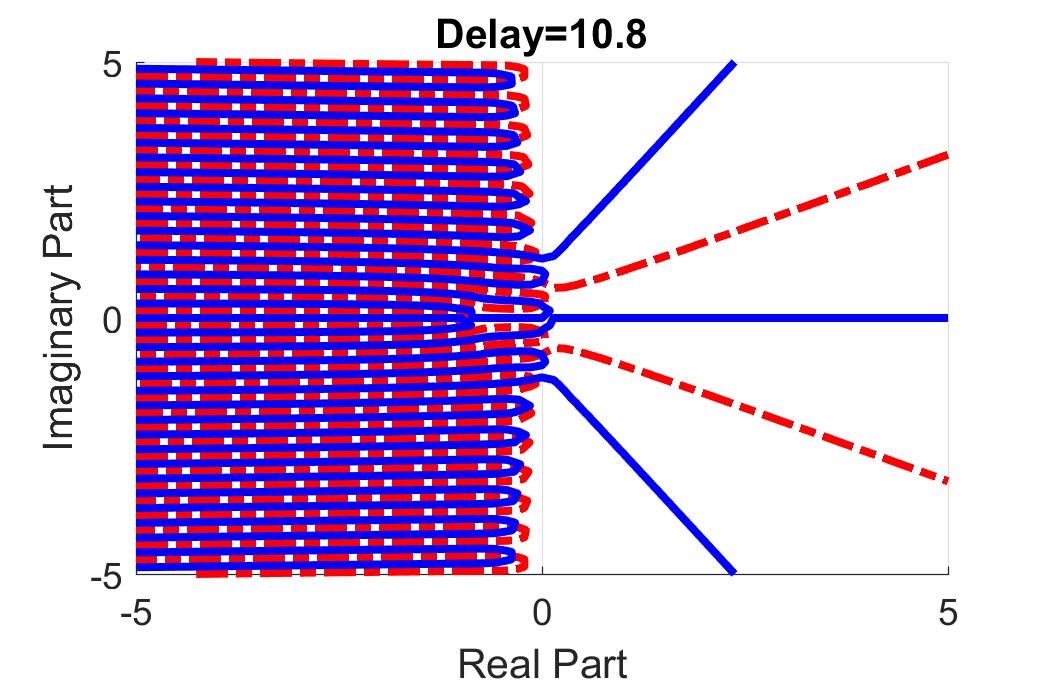}
		\end{subfigure}
		\caption{Contour plots for different values of delay $\tau$ showing stability switches.}\label{fig:delay2}
	\end{figure}
	As a result the following statement holds.
	\begin{lemma} The system (\ref{eq-1})--(\ref{eq-3}) with delay and $p_3(p_6-K_2)+ 2 K_3 p_6 \neq0$ is stable for all $\tau\geqslant 0$ if $F(0)>0$ and $\Delta_0 \leqslant 0$ (where $\Delta_0$ is from \ref{d0}). The system has stability switches at some $\{\tau_j^n\}$ for every positive root $v_j$ of (\ref{eq:98}). Furthermore, if $A=0$, $p_2=0$ or $p_5=p_6=0$ then the delay has no affect on the stability of the system.
	\end{lemma}

	\begin{example}\label{ex:delay}
		This example illustrates the dynamics of eigenvalues with respect to the time delay for the following set of parameters $p_3 = 0.41$, $p_6 = 0.91$, $K_2=0.81$, and $K_3=0.41$ in the equation (\ref{eq:98}). Taking the real and imaginary parts, we rewrite the equation as a system
		\begin{eqnarray*}
			\left\{\begin{array}{rl}-K_2p_3&+K_3p_6+p_3p_6-K_2x+K_3x+p_3x-K_2p_3x+p_6x+p_3p_6x+x^2\\
				&-K_2x^2+p_3x^2+p_6x^2+x^3-y^2+K_2y^2-p_3y^2-p_6y^2-3xy^2\\
				&+e^{-\tau x}K_3p_6\cos(\tau y)+e^{-\tau x}K_3x\cos(\tau y)+e^{-\tau x}K_3y\sin(\tau y)=0,\\
				-K_2y&+K_3y+p_3y-K_2p_3y+p_6y+p_3p_6y+2xy-2K_2xy+2p_3xy\\
				&+2p_6xy+3x^2y-y^3+e^{-\tau x}K_3y\cos(\tau y)-e^{-\tau x}K_3p_6\sin(\tau y)\\
				&-e^{-\tau x}K_3x\sin(\tau y)=0.
			\end{array}\right.
		\end{eqnarray*}
		The red lines in Figure \ref{fig:delay2} represent the solution curves for the first equation and the blue lines in Figure \ref{fig:delay2} represent the solution curves for the second equation for different values of delay $\tau$. The eigenvalues, which are roots of (\ref{eq:98}), correspond to intersections between the red and blue lines. 
		
		When there is no delay, i.e. $\tau=0$, we only have three eigenvalues $\lambda \approx-0.9, -0.2\pm0.8i$ (see Figure \ref{fig:delay2}). When delay is is non-zero, countably many eigenvalues originate from $-\infty$ and move toward the imaginary axis as $\tau$ increases (see Figure \ref{fig:delay2}). The eigenvalues can cross the imaginary axis only at the points $y_1\approx\pm0.7$ and $y_2\approx\pm 0.25$ which are real roots of the equation (\ref{eq:101}) (see Figure \ref{fig:delay1}). The density of complex eigenvalues around these crossing points $y_1,y_2$ is increasing as the $\tau$ gets larger (see Figure \ref{fig:delay2}). 
		
		When the delay $\tau <\tau^*\approx2$ (where $\tau^*$ is a critical value found as a solution of (\ref{eq:103}) with $v_1=\sqrt{|y_1|}$) all eigenvalues are stable. The first stability switch happens at $\tau^*\approx2$ when two complex conjugate eigenvalues cross the imaginary axis at $y_1\approx\pm0.7$ changing the sign of the real part from negative to positive. At a later time $\tau^*\approx11$ (where this $\tau^*$ is a critical value found as a solution of (\ref{eq:103}) with $v_2=\sqrt{|y_2|}$) this complex pair  will cross the imaginary axis back changing the sign of the real part from positive to negative (see Figure \ref{fig:delay1}). 
		
		Solving (\ref{eq:103}) and taking into account the periodicity of the arctangent function one can obtained the infinite sequences of delays associated with $v_1$ and another infinite sequence associated with $v_2$ at which stability switches may happen. At time delays associated with $v_1$ a complex conjugate pair of eigenvalues may cross the imaginary axis from left to right and for time delays associated with $v_2$ the pair may cross the imaginary axis from right to left. If the derivative of $F(y)$ (see (\ref{eq:98})) does not change sign at the corresponding $\tau^*$ from either of the two sequences above, then the crossing of the imaginary axis does not happen. 
	\end{example}
	
	\begin{figure}
		\centering
		\begin{subfigure}[b]{0.45\textwidth}
			\includegraphics[width=\textwidth]{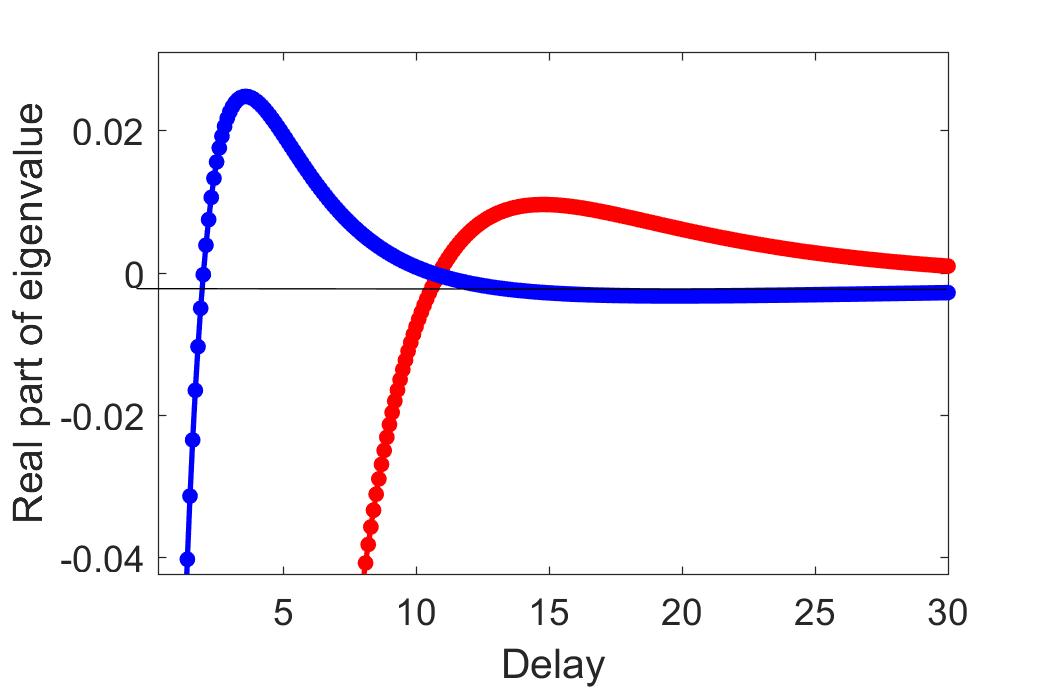}
		\end{subfigure}
		\hfill
		\begin{subfigure}[b]{0.45\textwidth}
			\includegraphics[width=\textwidth]{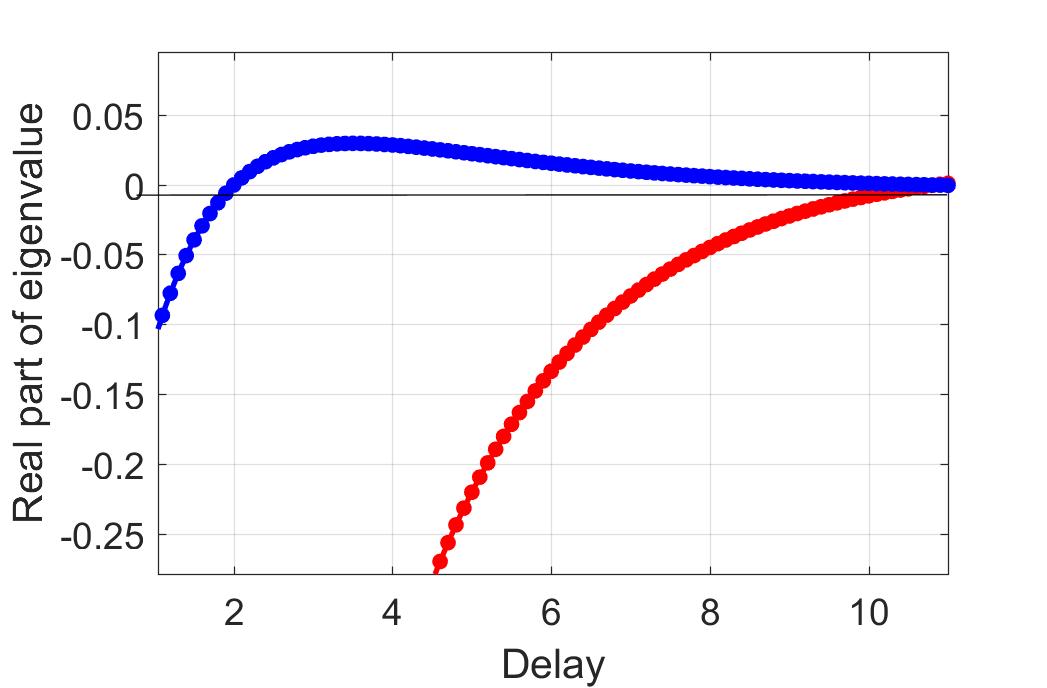}
		\end{subfigure}
		\vskip\baselineskip
		\begin{subfigure}[b]{0.45\textwidth}
			\includegraphics[width=\textwidth]{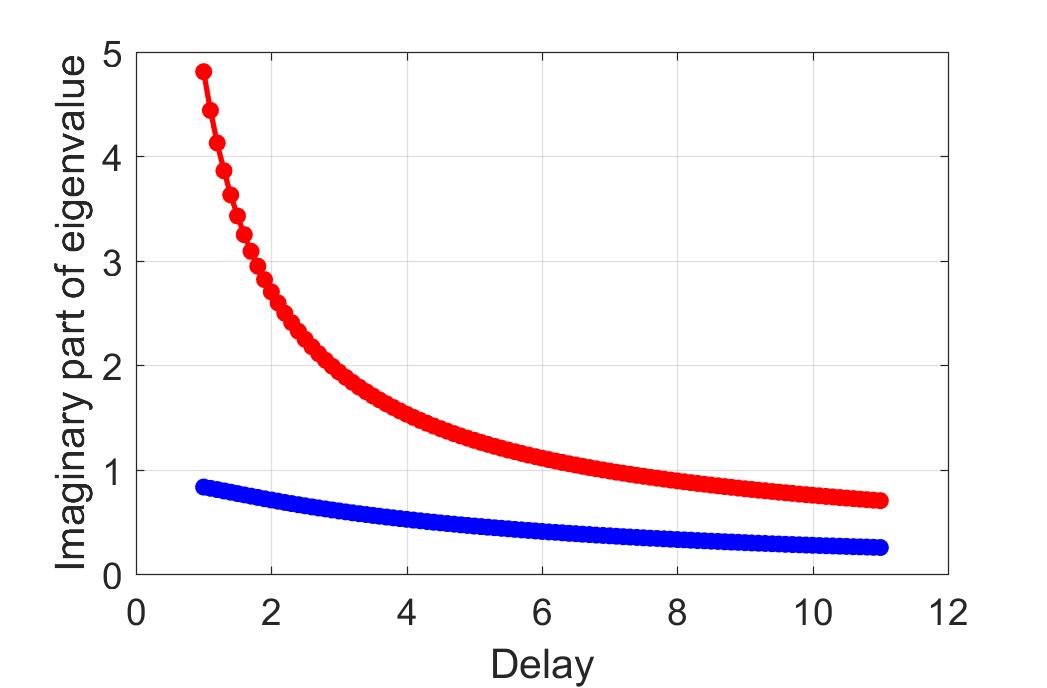}
		\end{subfigure}
		\quad
		\begin{subfigure}[b]{0.45\textwidth}
			\includegraphics[width=\textwidth]{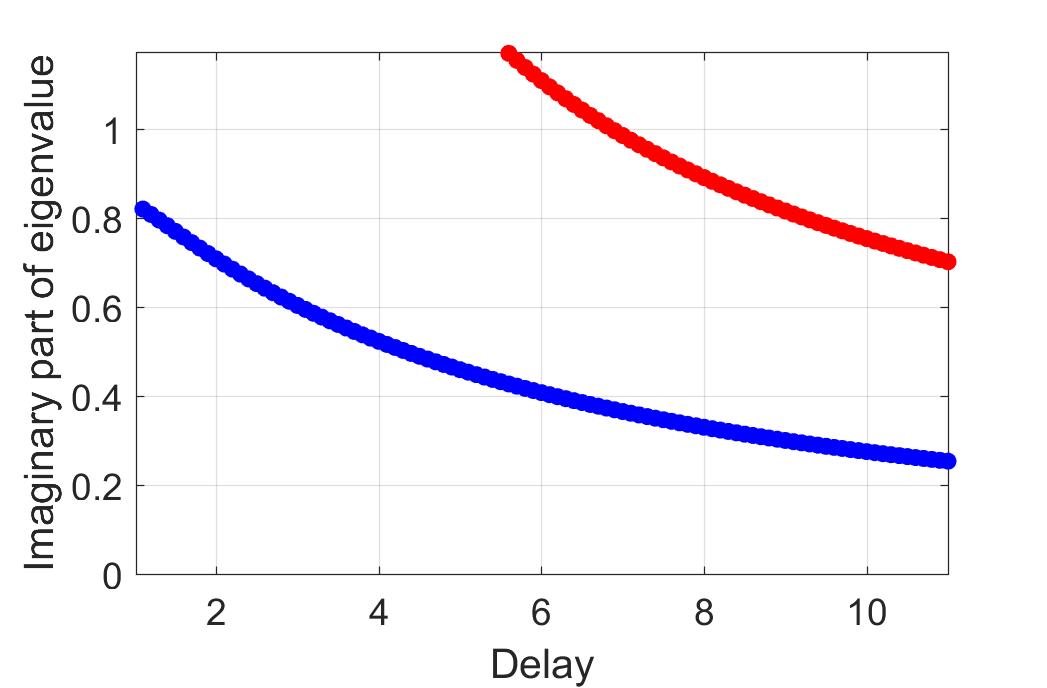}
		\end{subfigure}
		\caption{Tracking two complex eigenvalues to see how the value of their real part changes and how the value of their imaginary part changes.}\label{fig:delay1}
	\end{figure}
	
	%%%%%%%%%%%%%%%%%%%%
	
	\subsection{Global in time existence of solutions}
	
	In this section, by Picard's method we prove the existence of solutions to the problem (1.1)--(1.5).
	\begin{theorem}\label{Th-exist}
		If $A > 0$, $p_i > 0$,  and
		$$
		a'_{\tau}(0) + p_3 a_{\tau}(0) = \tfrac{A}{1+p_2 o_0 r_0},
		$$
		then the problem (\ref{a-1})--(\ref{smo-con}) has a unique non-negative solution
		$(a(t)$, $r(t)$, $o(t))$ in $C^2$ for all $t > 0$. Moreover, there exists
		a time $T^* > 0$ such that
		$$
		\tfrac{A p_6}{p_3 p_6 + A p_2(p_5+1)} \leqslant   o(t), \,
		a (t) \leqslant \tfrac{A}{p_3} , \ \tfrac{p_5}{p_6} \leqslant  r (t) \leqslant \tfrac{p_5 + 1}{p_6}
		\ \ \forall \, t \geqslant T^*.
		$$
	\end{theorem}
	
	For example, if $a_{\tau}(t) = a_0 + \Lambda \, t^2 e^{-t}$ then we get
	$$
	a_0 = \tfrac{A}{p_3(1+p_2 o_0 r_0)}, \ o_0 > 0, \ r_0 > 0 .  
	$$
	
	\begin{proof}[Proof of Theorem~\ref{Th-exist}]
		
		We will construct a solution to (\ref{a-1})--(\ref{smo-con}) by the iterative  process.
		First of all, we will look for a solution on the interval $[0,\tau]$.
		From (\ref{a-3}) we obtain that
		\begin{multline} \label{o-1}
		o(t) = e^{-t} o_0 + e^{-t} \int \limits_0^t {  a_{\tau} (s-\tau)e^{ s} \,ds  } =  
		e^{-t} o_0 + e^{-(t-\tau)} \int \limits_{-\tau}^{t-\tau} {  a_{\tau} (s )e^{ s} \,ds  } =: o_1(t) \ \forall \, t \in [0,\tau].
		\end{multline}
		So, by (\ref{smo-con}) and (\ref{o-1}) we have $  o(t) \in C^2[0,\tau]$,  and
		\begin{multline}\label {o-2}
		\underline{o}_1(t) := e^{-t} o_0 + (1 - e^{-t})
		\mathop {\min} \limits_{[-\tau,0]} a_{\tau}(t) \leqslant o(t) \leqslant  
		\overline{o}_1(t) := e^{-t} o_0 + (1 - e^{-t})
		\mathop {\max} \limits_{[-\tau,0]} a_{\tau}(t)
		\end{multline}
		for all $t \in [0,\tau]$. Integrating (\ref{a-1}) and (\ref{a-2}) on the interval $[0,\tau]$, taking into account
		(\ref{o-1}), we arrive at
		\begin{equation}\label{o-3}
		a(t) = e^{-p_3 t } a_{\tau}(0) + A \,e^{-p_3 t} \int \limits_0^t { \tfrac{  e^{p_3 s} \,ds}{ 1 + p_2 r(s) o_1(s)}  } ,
		\end{equation}
		\begin{equation}\label{o-4}
		r(t) = e^{-p_6 t } r_0  - p_4 \,e^{-p_6 t} \int \limits_0^t { \tfrac{  e^{p_6 s} \,ds}{ p_4 +   r^2(s) o^2_1(s)}  }
		+ \tfrac{p_5 + 1}{p_6} (1 - e^{-p_6 t})
		\end{equation}
		for all $t \in [0,\tau]$. By (\ref{o-3}), (\ref{o-4}) we find that
		\begin{equation}\label{o-5}
		\underline{r}_1(t) := e^{-p_6 t } r_0 + \tfrac{p_5}{p_6} (1 - e^{-p_6 t}) \leqslant r(t) \leqslant
		\overline{r}_1(t):= e^{-p_6 t } r_0 + \tfrac{p_5 + 1}{p_6} (1 - e^{-p_6 t}),
		\end{equation}
		\begin{multline}\label{o-6}
		\underline{a}_1(t) := e^{-p_3 t } a_{\tau}(0) + \tfrac{A}{p_3[1 + p_2 \mathop {\max} \limits_{[ 0, \tau]} \overline{o}_1 (t)
			\mathop {\max} \limits_{[ 0, \tau]} \overline{r}_1 (t) ]} (1 - e^{-p_3 t}) \leqslant a(t) \leqslant \\
		\overline{a}_1(t):= e^{-p_3 t } a_{\tau}(0) + \tfrac{A}{p_3} (1 - e^{-p_3 t})
		\end{multline}
		for all $t \in [0,\tau]$. As a result, estimates (\ref{o-2}), (\ref{o-5}), (\ref{o-6})
		imply positivity of $o(t),\,a(t),\,r(t)$ on $[0,\tau]$ provided
		$o_0 > 0$, $r_0 > 0$, and $a_{\tau}(0) > 0$.  As the right-hand side of (\ref{o-4}) is Lipschitz
		continuous on $r$ then there exists a  unique solution of (\ref{o-4}) on whole interval
		$[0,\tau]$ and, as a result, the one of  (\ref{o-3}).  Moreover, obviously the solution
		$o(t),\,a(t),\,r(t) \in C^2 [0,\tau]$
		if the following fitting condition  is true:
		
		\begin{equation}\label{fit}
		a'_{\tau}(0) + p_3 a_{\tau}(0) = \tfrac{A}{1+p_2 o_0 r_0}.  
		\end{equation}

		Let us denote the corresponding solution to (\ref{o-1}), (\ref{o-5}), (\ref{o-6})
		on $[0,\tau]$ by $ (o_1(t), a_1(t), r_1(t))$. Now we will find a solution
		on the interval $[\tau, 2\tau]$. By (\ref{a-1})--(\ref{a-3}) we get
		\begin{multline} \label{o-7}
		o(t) = e^{-(t-\tau) } o_1 (\tau) + e^{-t} \int \limits_{\tau}^t {  a(s-\tau)e^{ s} \,ds  } =
		e^{-(t-\tau) } o_1(\tau) +  \\
		e^{-(t-\tau)} \int \limits_{0}^{t-\tau} {  a_{1} (s )e^{ s} \,ds  } =
		e^{- t } o_0 + e^{-(t-\tau)} \Bigl[ \int \limits_{-\tau}^{0} {  a_{\tau} (s )e^{ s} \,ds  } + 
		\int \limits_{0}^{t-\tau} {  a_{1} (s )e^{ s} \,ds  } \Bigr]  =: o_2(t) ,
		\end{multline}
		\begin{equation}\label{o-8}
		a(t) = e^{-p_3 ( t -\tau) } a_{1}(\tau) + A \,e^{-p_3 t} \int \limits_{\tau}^t { \tfrac{  e^{p_3 s} \,ds}{ 1 + p_2 r(s) o_2(s)}  },
		\end{equation}
		\begin{equation}\label{o-9}
		r(t) = e^{-p_6 (t -\tau) } r_1(\tau)  - p_4 \,e^{-p_6 t} \int \limits_{\tau}^t { \tfrac{  e^{p_6 s} \,ds}{ p_4 +   r^2(s) o^2_2(s)}  }
		+ \tfrac{p_5 + 1}{p_6} (1 - e^{-p_6 ( t - \tau)})
		\end{equation}
		for all $t \in [\tau,2\tau]$. The system (\ref{o-7})--(\ref{o-9}) has a  unique solution
		$o_2(t),\,a_2(t),\,r_2(t) \in C^2 [\tau,2\tau]$.  Moreover,
		\begin{multline}\label {o-2-2}
		\underline{o}_2(t) := e^{-t} o_0 + e^{-t} (e^{\tau} -1)
		\mathop {\min} \limits_{[-\tau,0]} a_{\tau}(t)  +  (1 - e^{-(t -\tau)})
		\mathop {\min} \limits_{[0,\tau]} a_{1}(t) \leqslant o(t) \leqslant  \\
		\overline{o}_2(t) := e^{-t} o_0 + e^{-t} (e^{\tau} -1)
		\mathop {\max} \limits_{[-\tau,0]} a_{\tau}(t)  +  (1 - e^{-(t -\tau)})
		\mathop {\max} \limits_{[0,\tau]} a_{1}(t),
		\end{multline}
		\begin{multline} \label{o-5-2}
		\underline{r}_2(t) := e^{-p_6 (t -\tau) } r_1(\tau) + \tfrac{p_5}{p_6} (1 - e^{-p_6 (t-\tau)}) \leqslant r(t) \leqslant   \\
		\overline{r}_2(t):= e^{-p_6 (t -\tau) } r_1(\tau) + \tfrac{p_5 + 1}{p_6} (1 - e^{-p_6 (t -\tau)}),
		\end{multline}
		\begin{multline}\label{o-6-2}
		\underline{a}_2(t) := e^{-p_3 (t -\tau) } a_{1}(\tau) a+ \tfrac{A}{p_3[1 + p_2 \mathop {\max} \limits_{[ \tau,2\tau]} \overline{o}_2 (t)
			\mathop {\max} \limits_{[ \tau,2\tau]} \overline{r}_2 (t) ]} (1 - e^{-p_3 (t-\tau)}) \leqslant \\
		a(t) \leqslant\overline{a}_2(t):= e^{-p_3 (t-\tau) } a_{1}(\tau ) + \tfrac{A}{p_3} (1 - e^{-p_3 (t-\tau)})
		\end{multline}
		for all $t \in [\tau,2\tau]$. Continuing this iteration procedure, we derive
		\begin{equation} \label{o-7-k}
		o(t) = e^{-(t-(k-1)\tau) } o_{k-1} ((k-1)\tau) +
		e^{-(t-\tau)} \int \limits_{(k-2)\tau}^{t-\tau} {  a_{k-1}(s )e^{ s} \,ds  }
		=: o_k(t) ,
		\end{equation}
		\begin{equation}\label{o-8-k}
		a(t) = e^{-p_3 ( t -(k-1)\tau) } a_{k-1}((k-1)\tau) + A \,e^{-p_3 t} \int \limits_{(k-1)\tau}^t { \tfrac{  e^{p_3 s} \,ds}{ 1 + p_2 r(s) o_k(s)}  },
		\end{equation}
		\begin{multline} \label{o-9-k}
		r(t) = e^{-p_6 (t -(k-1)\tau) } r_{k-1}((k-1)\tau)  - p_4 \,e^{-p_6 t} \int \limits_{(k-1)\tau}^t { \tfrac{  e^{p_6 s} \,ds}{ p_4 +   r^2(s) o^2_k(s)}  }
		+  
		\tfrac{p_5 + 1}{p_6} (1 - e^{-p_6 ( t - (k-1)\tau)})
		\end{multline}
		for all $t \in [(k-1) \tau, k\,\tau ]$, $k \in \mathbb{N}$, where $a_0(t) = a_{\tau}(t)$. This system has a  unique solution
		$o_k(t)$, $a_k(t)$, $r_k(t)$  $\in C^{k+1} [(k-1) \tau, k\,\tau ]$. Moreover,
		\begin{multline}\label {o-2-k}
		\underline{o}_k(t) :=  e^{-(t-(k-1)\tau) } o_{k-1} ((k-1)\tau) + 
		(1 - e^{-(t -(k-1)\tau)})
		\mathop {\min} \limits_{[(k-2)\tau,(k-1)\tau]} a_{k-1}(t) \leqslant o(t) \leqslant  \\
		\overline{o}_k(t) :=  e^{-(t-(k-1)\tau) } o_{k-1} ((k-1)\tau) +
		(1 - e^{-(t -(k-1)\tau)})
		\mathop {\max} \limits_{[(k-2)\tau,(k-1)\tau]} a_{k-1}(t),
		\end{multline}
		\begin{multline} \label{o-5-k}
		\underline{r}_k(t) := e^{-p_6 (t -(k-1)\tau) } r_{k-1}((k-1)\tau) + \tfrac{p_5}{p_6} (1 - e^{-p_6 (t-(k-1)\tau)}) \leqslant r(t) \leqslant   \\
		\overline{r}_k(t):= e^{-p_6 (t -(k-1)\tau) } r_{k-1}((k-1)\tau) + \tfrac{p_5 + 1}{p_6} (1 - e^{-p_6 (t -(k-1)\tau)}),
		\end{multline}
		\begin{multline}\label{o-6-k}
		\underline{a}_k(t) := e^{-p_3 (t -(k-1)\tau) } a_{k-1}((k-1)\tau) + 
		\tfrac{A}{p_3[1 + p_2 \mathop {\max} \limits_{[ (k-1)\tau,k\tau]} \overline{o}_k (t)
			\mathop {\max} \limits_{[ (k-1)\tau, k \tau]} \overline{r}_k (t) ]} (1 - e^{-p_3 (t-(k-1)\tau)}) \leqslant \\
		a(t) \leqslant\overline{a}_k(t):= e^{-p_3 (t-(k-1)\tau) } a_{k-1}((k-1)\tau ) + \tfrac{A}{p_3} (1 - e^{-p_3 (t-(k-1)\tau)})
		\end{multline}
		for all $t \in [(k-1) \tau, k\,\tau ]$. As a result, the problem (\ref{a-1})--(\ref{in-con}) has unique global in time solution,
		and, letting $k \to +\infty$, we get
		$$
		\tfrac{A p_6}{p_3 p_6 + A p_2(p_5+1)} \leqslant \mathop {\lim} \limits_{k \to +\infty} o_k(t), \
		\mathop {\lim} \limits_{k \to +\infty} a_k(t) \leqslant \tfrac{A}{p_3} ,
		$$
		$$
		\tfrac{p_5}{p_6} \leqslant \mathop {\lim} \limits_{k \to +\infty} r_k(t) \leqslant \tfrac{p_5 + 1}{p_6}.
		$$
	\end{proof}
	
	%%%%%%%%%%%%%
	
	\subsection{Existence of periodic solutions}

	\begin{theorem}\label{Th-periodic}
		Under the conditions of Theorem~\ref{Th-exist}, the system (\ref{a-1})--(\ref{smo-con}) has at least one
		$C^2$-smooth $T$-periodic solution, where $T\neq \tau$.
	\end{theorem}
	
	\begin{proof}[Proof of Theorem~\ref{Th-periodic}]
		
		The main line of proof follows (see \cite[pp. 278--280]{Krasn1}) (see also \cite[Theorem~5]{Krasn2}).
		Rewrite the system of (\ref{a-1})--(\ref{a-3}) in the following form
		\begin{equation}\label{s-1}
		\textbf{x}'(t) = M \,\textbf{x}(t) + B \, \textbf{x}(t -\tau) + \textbf{f}(\textbf{x}(t)),
		\end{equation}
		where
		$$
		\textbf{x}(t) = \left (\begin{array}{c}
		a(t) \\
		r(t) \\
		o(t)
		\end{array} \right), \
		\textbf{f}(\textbf{x}(t)) = \left (\begin{array}{c}
		\tfrac{A}{1+ p_2 o r} \\
		p_5 + \tfrac{(or)^2}{p_4+  (o r)^2} \\
		0
		\end{array} \right),
		$$
		$$
		M = \left(
		\begin{array}{ccc}
		-p_3 & 0 & 0 \\
		0 & -p_6 & 0 \\
		0 & 0 & -1 \\
		\end{array}
		\right), \
		B = \left(
		\begin{array}{ccc}
		0 & 0 & 0 \\
		0 & 0 & 0 \\
		1 & 0 & 0 \\
		\end{array}
		\right).
		$$
		Obviously, the right-hand side of (\ref{s-1}) is $T$-periodic with respect to $t$ as it does not
		depend on time explicitly. Without loss of generality, we may assume that
		$$
		0 \leqslant \tau < T.
		$$
		This is true because otherwise we could represent the $\tau$ in the form
		$$
		0 < \tau = n T + \tau_1, \text{ where } n \in \mathbb{Z}^+,\ \tau_1 \in [0,T).
		$$
		Then shift to the auxilliary equation
		$$
		\textbf{x}'(t) = M \,\textbf{x}(t) + B \, \textbf{x}(t -\tau_1) + \textbf{f}(\textbf{x}(t))
		$$
		the $T$-periodic solutions of which coincide with the $T$-periodic solutions of (\ref{s-1}).
		
		On the set of all vector-valued functions $\textbf{x}(t)$ defined on $[0,T]$, let us
		define an operator $S_{\tau}$ by
		$$
		S_{\tau} \textbf{x}(t) :=
		\left \{ \begin{gathered}
		\textbf{x}(t -\tau) \text{ if } \tau \leqslant t \leqslant T, \\
		\textbf{x}( t-\tau + T) \text{ if } 0 \leqslant t < \tau .
		\end{gathered} \right.
		$$
		Note that the  $T$-periodic solutions of (\ref{s-1}) coincides with the solutions of
		the following integral equations:
		\begin{equation}\label{s-2}
		\textbf{x}(t) = T(\tau,  \textbf{x}):= \textbf{x}(0) + \int \limits_0^t {( M \,\textbf{x}(s) + B \, S_{\tau} \textbf{x}(s) + \textbf{f}(\textbf{x}(s)))\,ds }.
		\end{equation}
		The operator $T(\tau,  \textbf{x})$ maps every continuous vector-valued function $\textbf{x}(t)$
		into a continuous vector-valued function for $0 \leqslant t \leqslant T$, therefore  $T(\tau,  \textbf{x})$
		is compact in $C$. Next, we will show that for all  $T$-periodic solutions $\textbf{x}_p(t) $ there exists $R > 0$ such that
		\begin{equation}\label{s-3}
		|  \textbf{x}_p(t) | \leqslant R < \infty.
		\end{equation}
		Really, from (\ref{s-2}) we deduce that
		\begin{multline*}
		| \textbf{x}_p(t)|  \leqslant | \textbf{x}_p (0)| + \int \limits_0^t {[ |M| \,|\textbf{x}_p(s)| + |B| \, |S_{\tau} \textbf{x}_p(s)| + |\textbf{f}(\textbf{x}_p(s))| ] \,ds } \leqslant    \\
		| \textbf{x}_p(0)| + [(p_3^2 + p_6^2 + 1)^{\frac{1}{2}}  +1]  \int \limits_0^t { |\textbf{x}_p(s)| \,ds } +
		(A^2 + (p_5 + 1)^2 )^{\frac{1}{2}} t.
		\end{multline*}
		From here, using Gr\"{o}nwall's lemma, we arrive at
		$$
		| \textbf{x}_p(t)|  \leqslant ( | \textbf{x}_p(0)| + a ) e^{b T} - a,
		$$
		where $a = \frac{(A^2 + (p_5 + 1)^2 )^{\frac{1}{2}}}{(p_3^2 + p_6^2 + 1)^{\frac{1}{2}}  +1}$, $b = (p_3^2 + p_6^2 + 1)^{\frac{1}{2}}  +1$.
		Hence, (\ref{s-3}) holds with $R = ( | \textbf{x}_p(0)| + a ) e^{b T} - a$.
		As a result, by the fixed point theorem the integral equation (\ref{s-2}) has at least one solution,
		and consequently the equation (\ref{s-1})  has at least one $T$-periodic solution.
	\end{proof}
	%%%%%%%%%%%%%%%%%%%%
	\subsection{Periodic solutions with the period $T = \tau$}
	
	\begin{lemma}\label{Th-tau}
		If
		$$
		o_0  =  a_{\tau}(-\tau),\ a_{\tau}(0) = \tfrac{A} { p_3 }\Bigl[1 + p_2 \sqrt{\tfrac{p_4(p_6 r_0 - p_5)}{p_5 + 1 - p_6 r_0}} \Bigr]^{-1},\
		\tfrac{p_5}{p_6} \leqslant r_0 \leqslant \tfrac{p_5+1}{p_6}.
		$$
		then the problem (\ref{a-1})--(\ref{in-con}) has at least one $\tau$-periodic solution.
	\end{lemma}
	
	\begin{example}\label{ex:periodic1}
		Let $A=1$, $p_2=11$, $p_3=1.2$, $p_4=0.05$, $p_5=0.11$, and $p_6=2.9$. Then the initial conditions are $a_0=\frac{A}{p_3}\left(1+p_2\sqrt{\frac{p_4(p_6 r_0-p_5)}{p_5+1-p_6 r_0}} \right)^{-1}$, $r_0=\frac{1}{2}\left( \frac{p_5}{p_6}+\frac{p_5+1}{p_6}\right)$, and $o_0=a_0$. With these parameter values we solve (\ref{a-1})--(\ref{in-con}) numerically using the Matlab solver dde23 \cite{Matlab}. The resulting periodic solutions can be seen in Figure \ref{fig:periodic1}.
		
		If we perturb the parameters by a bit, the periodicity changes. We illustrate a periodicity change by using the parameters $A=1$, $p_2=7$, $p_3=1.2$, $p_4=0.05$, $p_5=0.51$, and $p_6=3.1$. Then the initial conditions are $a_0=\frac{A}{p_3}\left(1+p_2\sqrt{\frac{p_4(p_6 r_0-p_5)}{p_5+1-p_6 r_0}} \right)^{-1}$, $r_0=\frac{1}{2}\left( \frac{p_5}{p_6}+\frac{p_5+1}{p_6}\right)$, and $o_0=a_0$. The resulting periodic solutions can be seen in Figure \ref{fig:periodic2}.
		
		Periodicity of solutions can also be illustrated by plotting delayed function versus no delay function or function versus derivative as seen in Figure \ref{fig:periodicity}. 
		
		\begin{figure}
			\centering
			\begin{subfigure}[b]{0.45\textwidth}
				\includegraphics[width=\textwidth]{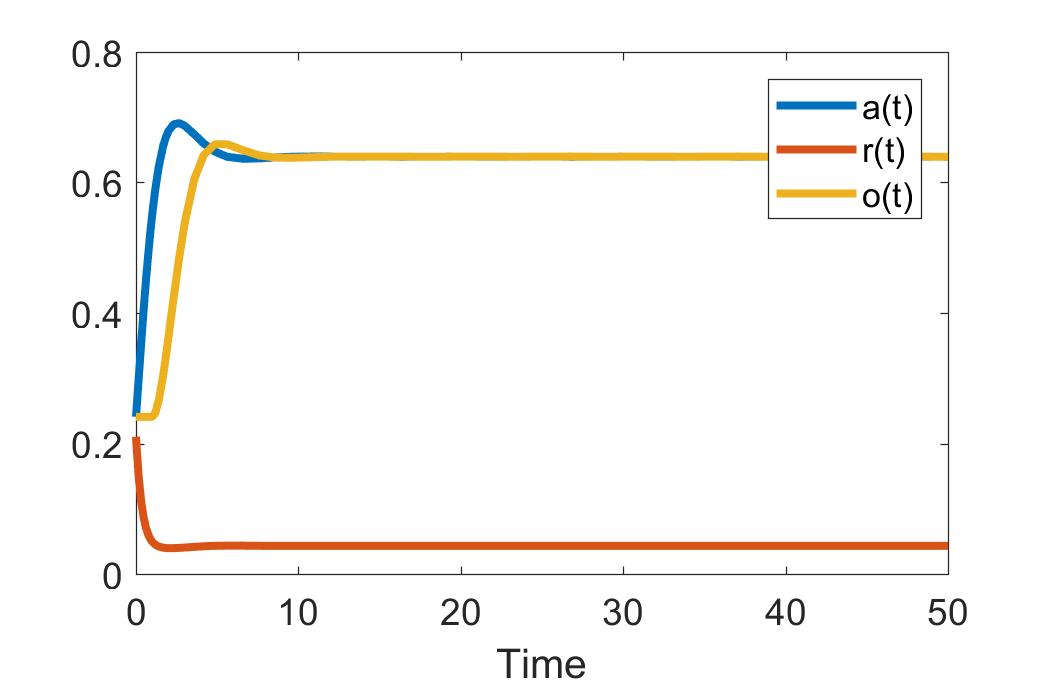}
				\caption{Regular Zoom}
			\end{subfigure}
			\begin{subfigure}[b]{0.45\textwidth}
				\includegraphics[width=\textwidth]{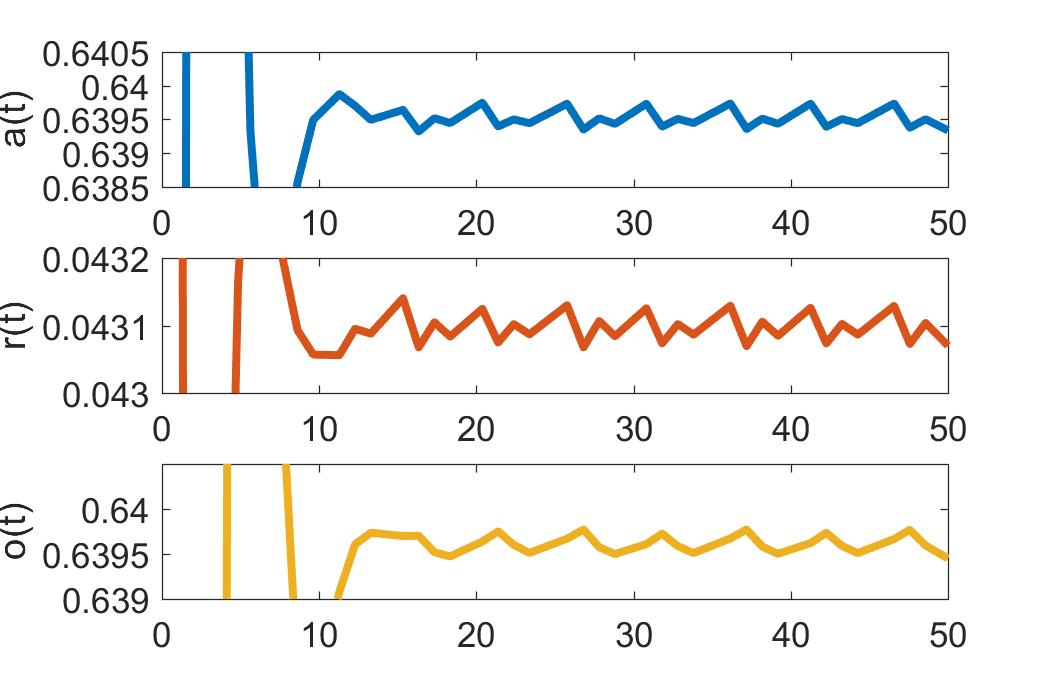}
				\caption{Zoomed in}
			\end{subfigure}
			\caption{Plot of the solutions $a(t)$, $r(t)$, and $o(t)$ with the parameter values in Example \ref{ex:periodic1}. }\label{fig:periodic1}
		\end{figure}
		
		\begin{figure}
			\centering
			\begin{subfigure}[b]{0.45\textwidth}
				\includegraphics[width=\textwidth]{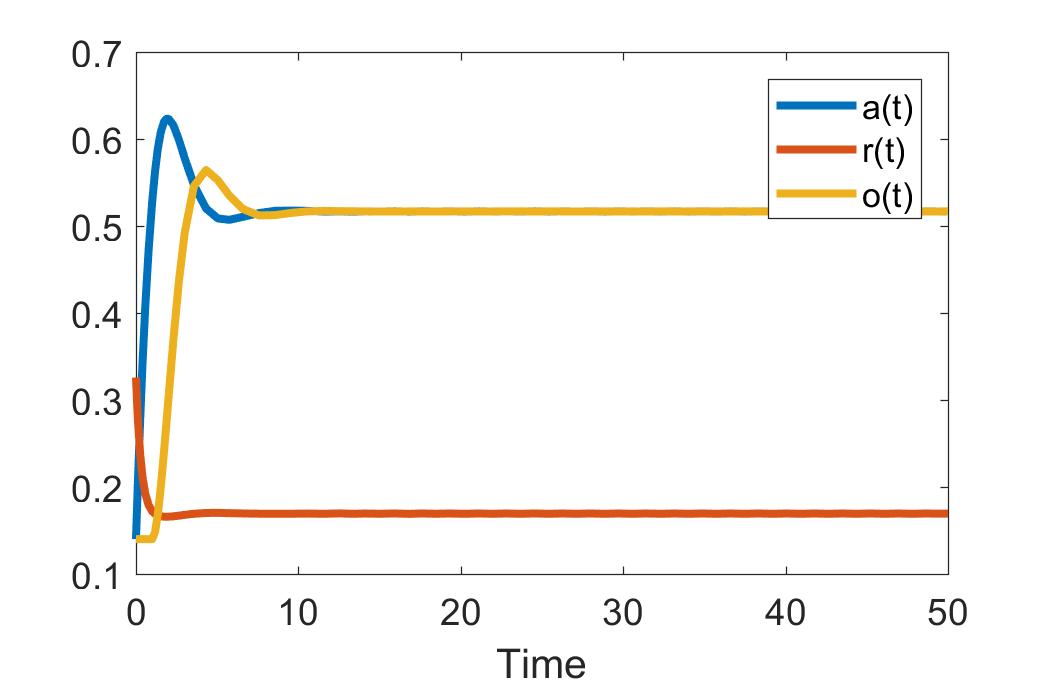}
				\caption{Regular Zoom}
			\end{subfigure}
			\begin{subfigure}[b]{0.45\textwidth}
				\includegraphics[width=\textwidth]{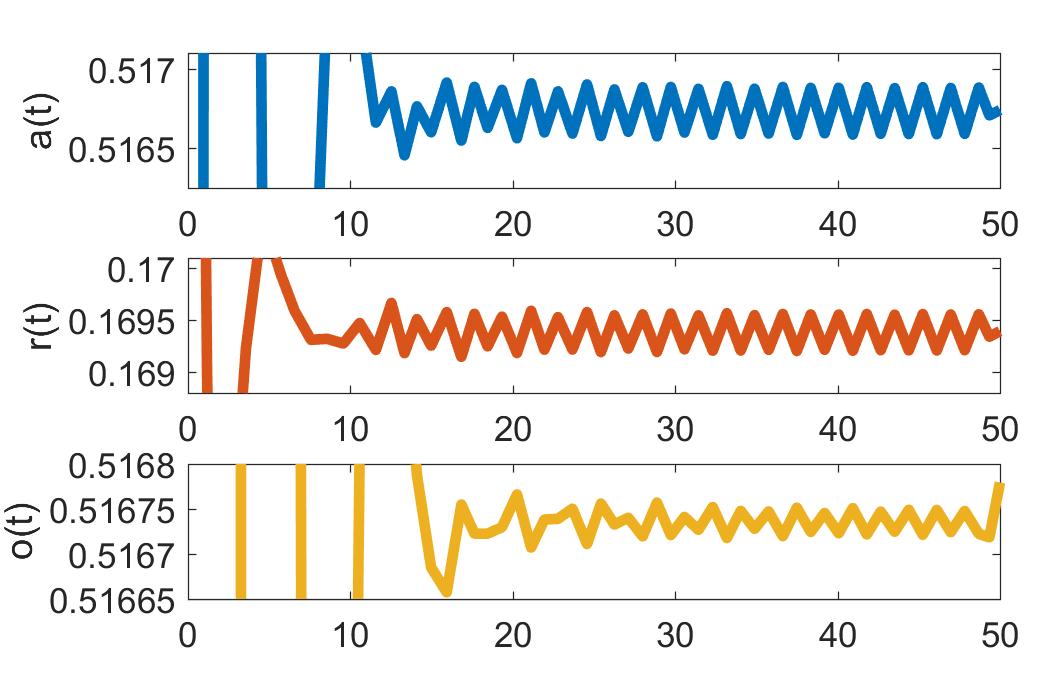}
				\caption{Zoomed in}
			\end{subfigure}
			\caption{Plot of the solutions $a(t)$, $r(t)$, and $o(t)$ with the parameter values in Example \ref{ex:periodic1} part b. }\label{fig:periodic2}
		\end{figure}
	\end{example}
	
	\begin{figure}
		\centering
		\begin{subfigure}[b]{0.45\textwidth}
			\includegraphics[width=\textwidth]{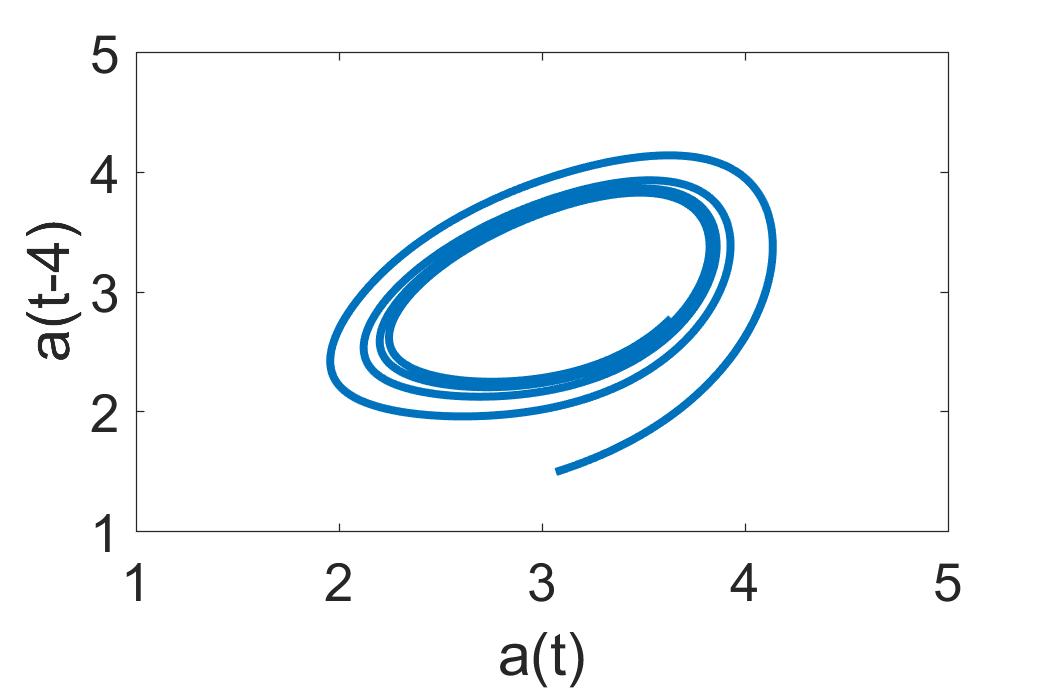}
			\caption{$a(t)$ and $a(t-4)$}
		\end{subfigure}
		\hfill
		\begin{subfigure}[b]{0.45\textwidth}
			\includegraphics[width=\textwidth]{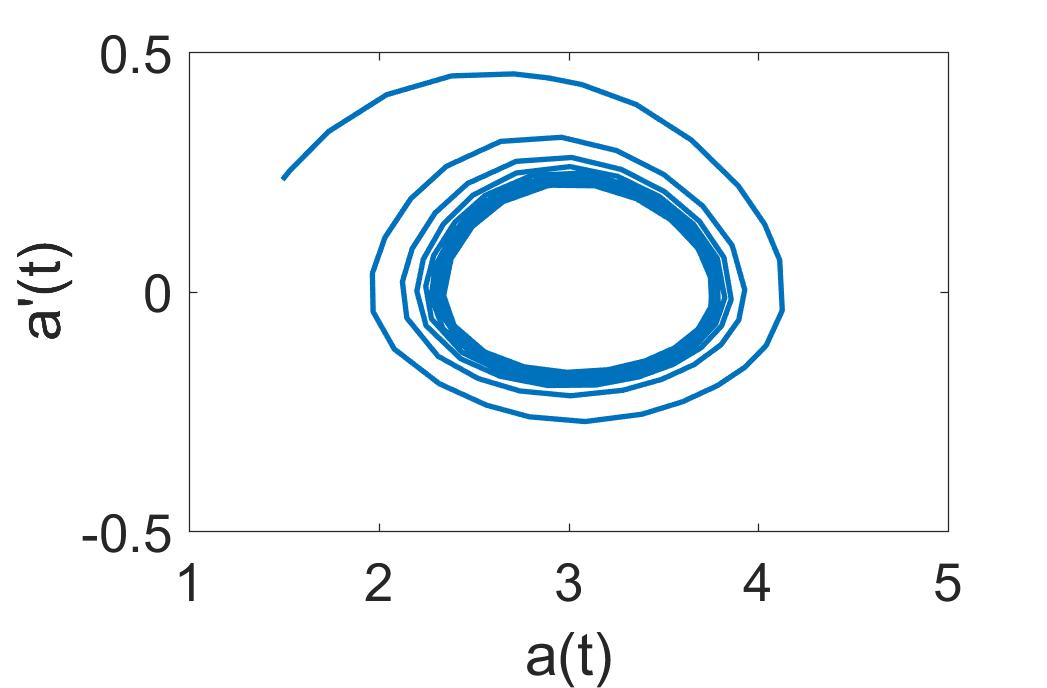}
			\caption{$a(t)$ and $a'(t)$}
		\end{subfigure}
		\vskip\baselineskip
		\begin{subfigure}[b]{0.45\textwidth}
			\includegraphics[width=\textwidth]{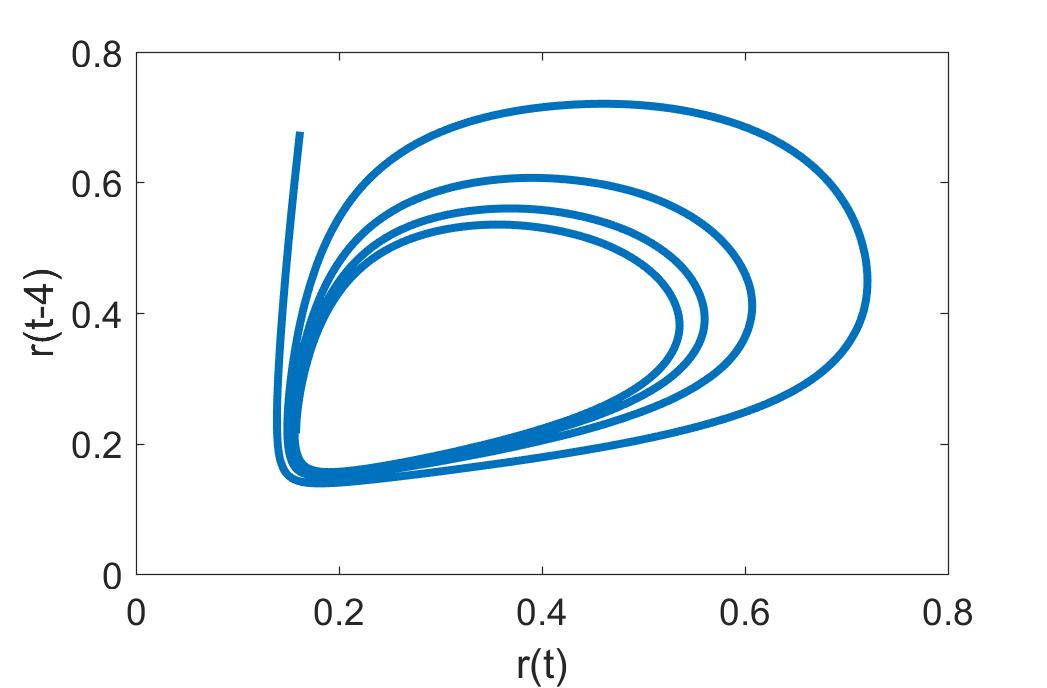}
			\caption{$r(t)$ and $r(t-4)$}
		\end{subfigure}
		\quad
		\begin{subfigure}[b]{0.45\textwidth}
			\includegraphics[width=\textwidth]{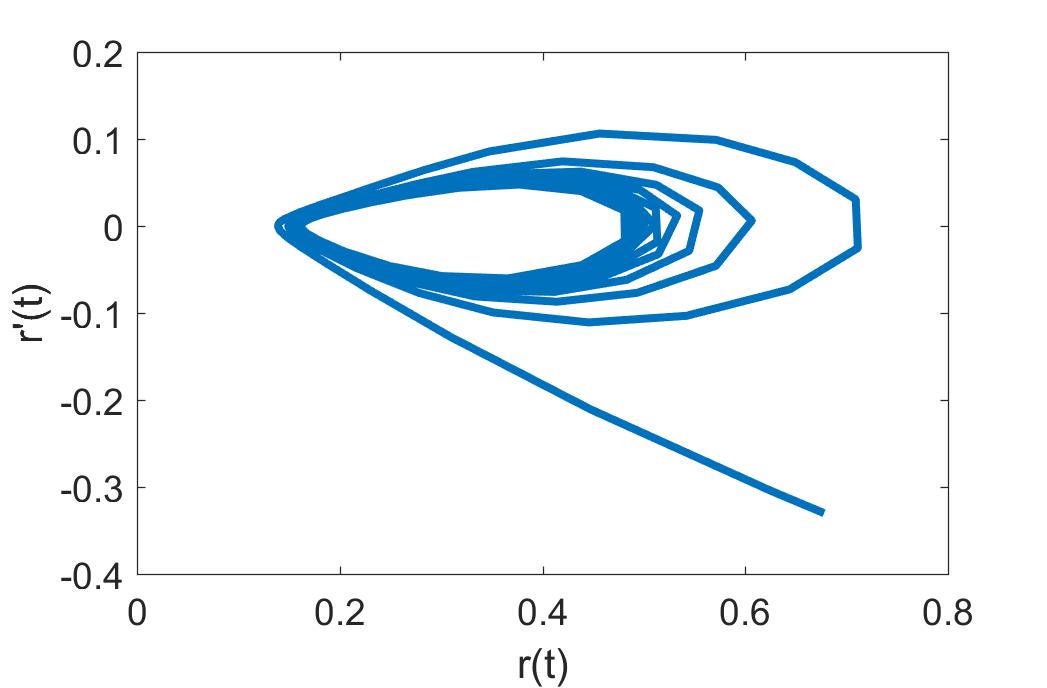}
			\caption{$r(t)$ and $r'(t)$}
		\end{subfigure}
		\vskip\baselineskip
		\begin{subfigure}[b]{0.45\textwidth}
			\includegraphics[width=\textwidth]{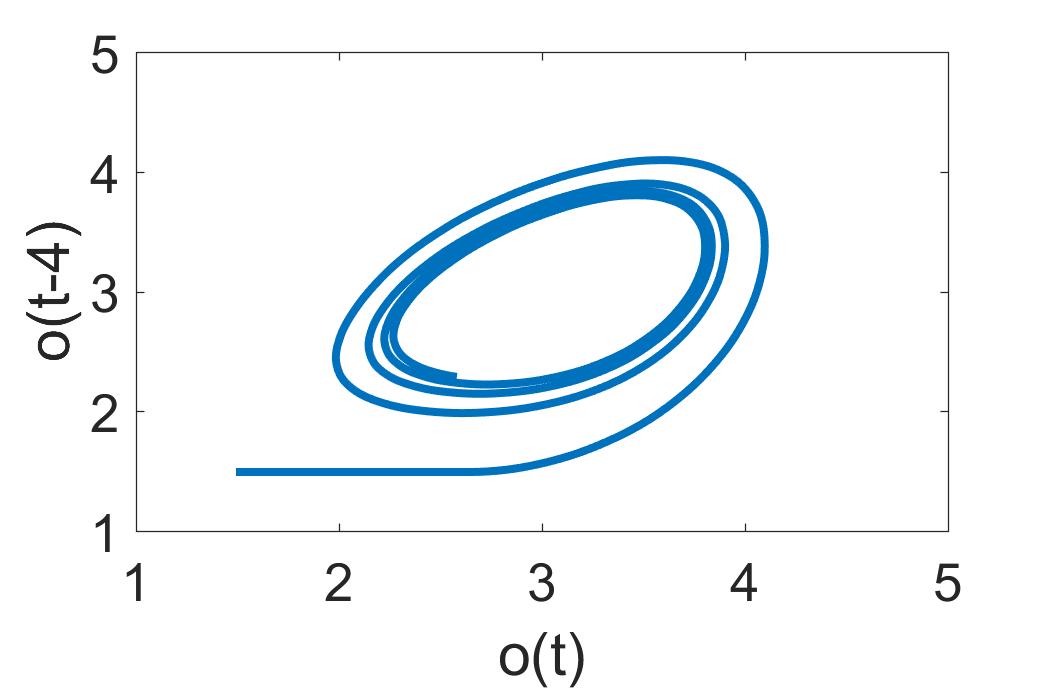}
			\caption{$o(t)$ and $o(t-4)$}
		\end{subfigure}
		\quad
		\begin{subfigure}[b]{0.45\textwidth}
			\includegraphics[width=\textwidth]{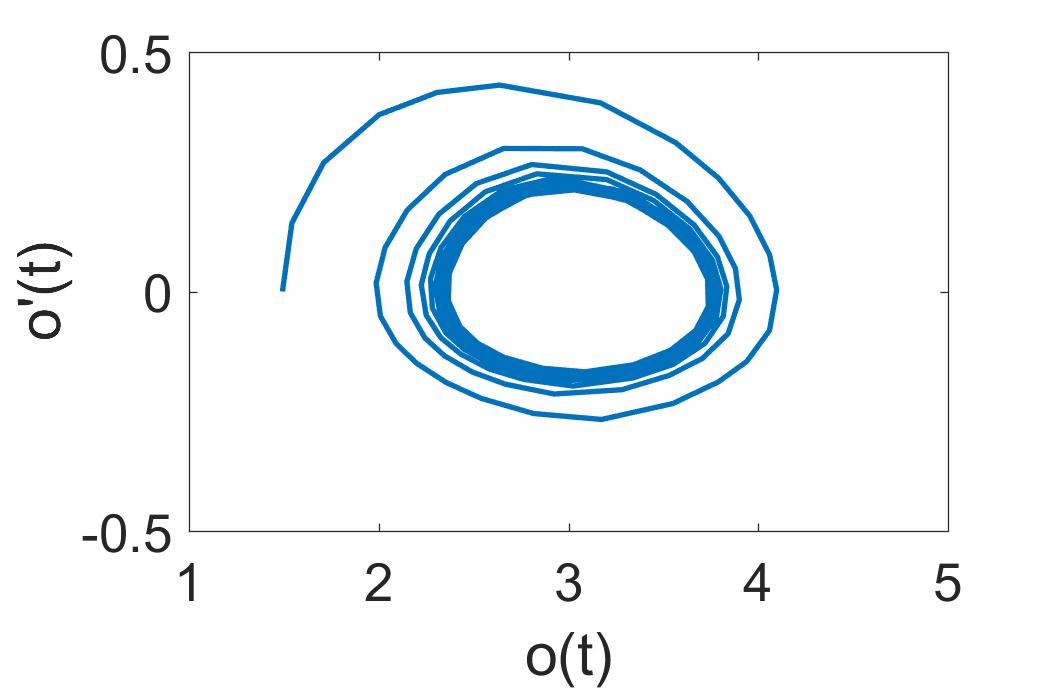}
			\caption{$o(t)$ and $o'(t)$}
		\end{subfigure}
		\caption{Plots using parameter values $A=1.5$, $p_2=1.8$, $p_3=0.2$, $p_4=5$, $p_5=0.11$, $p_6=0.9$ }\label{fig:lags}\label{fig:periodicity}
	\end{figure}
	
	\begin{proof}[Proof of Lemma~\ref{Th-tau}]
		
		Note that if a function $\Phi(t) \in C^2[a,b]$ is periodic with a period $T > 0$
		then $ \Phi(t+T) = \Phi (t)$ and $ \Phi'(t+T) = \Phi' (t)$.
		Let $(a(t),r(t),o(t))$ be a $T$-periodic solution of (\ref{a-1})--(\ref{a-3}) from Theorem~\ref{Th-periodic}.
		Then for this solution we have
		\begin{equation}\label{a-1-p}
		\tfrac{A}{1+p_2 o(t)r(t) } - p_3 a(t) = \tfrac{A}{1+p_2 o(t+T)r(t+T) } - p_3 a(t+T),
		\end{equation}
		\begin{equation}\label{a-2-p}
		- \tfrac{p_4}{p_4 + (o(t)r(t))^2} + 1 +  p_5 - p_6 r(t) = - \tfrac{p_4}{p_4 + (o(t+T)r(t+T))^2} + 1 +  p_5 - p_6 r(t+T),
		\end{equation}
		\begin{equation}\label{a-3-p}
		a(t -\tau) - o(t) =  a(t + T -\tau) - o(t+T),
		\end{equation}
		whence
		\begin{equation}\label{a-4-p}
		o(t)r(t) = o(t+T)r(t+T), \ \ a(t -\tau)   =  a(t + T -\tau) .
		\end{equation}
		Let $T = \tau > 0$. Then from (\ref{a-4-p}) for $t \in [0,\tau ]$ we have
		\begin{equation}\label{a-5-p}
		a_{\tau}(t) = a_1(t), \  \  o_1(t)r_1(t)= o_2(t+\tau) r_2(t +\tau)  .
		\end{equation}
		By (\ref{a-5-p}) at $t = 0$ we get
		$$
		a_{\tau}(0) = a_1(0) =e^{-p_3 \tau } a_{\tau}(0) + A \,e^{-p_3 \tau} \int \limits_0^{\tau} { \tfrac{  e^{p_3 s} \,ds}{ 1 + p_2 r_1(s) o_1(s)}  } ,
		$$
		$$
		o_0 = o_1(\tau)  =   e^{- \tau } o_0  + \int \limits_{-\tau}^0 { a_{\tau}(s) e^s \,ds}  ,
		$$
		$$
		r_0 =  r_1(\tau) = e^{-p_6 \tau } r_0  - p_4 \,e^{-p_6 \tau} \int \limits_0^{\tau} { \tfrac{  e^{p_6 s} \,ds}{ p_4 +   r^2_1(s) o^2_1(s)}  }
		+ \tfrac{p_5 + 1}{p_6} (1 - e^{-p_6 \tau}),
		$$
		whence
		$$
		a_{\tau}(0) =  \tfrac{A \,e^{-p_3 \tau}}{1 - e^{-p_3 \tau }} \int \limits_0^{\tau} { \tfrac{  e^{p_3 s} \,ds}{ 1 + p_2 r_1(s) o_1(s)}  } ,\
		o_0 =   \tfrac{1}{1 - e^{- \tau }} \int \limits_{-\tau}^0 { a_{\tau}(s) e^s \,ds},
		$$
		$$
		r_0 = \tfrac{ e^{-p_6 \tau}}{1 - e^{-p_6 \tau }} \Bigl[ - p_4  \int \limits_0^{\tau} { \tfrac{  e^{p_6 s} \,ds}{ p_4 +   r^2_1(s) o^2_1(s)}  }
		+ \tfrac{p_5 + 1}{p_6} ( e^{ p_6 \tau} - 1) \Bigr].
		$$
		Hence
		$$
		f_1(\tau) := (e^{ p_3 \tau} -1) \tfrac{a_{\tau}(0)}{A}  -  \int \limits_0^{\tau} { \tfrac{  e^{p_3 s} \,ds}{ 1 + p_2 r_1(s) o_1(s)}  } = 0,
		$$
		$$
		f_2(\tau) := (1 - e^{- \tau })o_0  -    \int \limits_{-\tau}^0 { a_{\tau}(s) e^s \,ds} = 0,
		$$
		$$
		f_3(\tau) : = \tfrac{1}{p_4} \Bigl[ \tfrac{p_5 + 1}{p_6} ( e^{ p_6 \tau} - 1) - ( e^{ p_6 \tau } -1) r_0 \Bigr] -  \int \limits_0^{\tau} { \tfrac{  e^{p_6 s} \,ds}{ p_4 +   r^2_1(s) o^2_1(s)}  } = 0.
		$$
		As $f_i(0) = 0$ and
		$$
		f'_1(\tau) = e^{ p_3 \tau} [ \tfrac{ p_3 a_{\tau}(0)}{A}  -    \tfrac{ 1 }{ 1 + p_2 r_0 o_0}  ] \equiv 0
		\text{ if } \tfrac{ p_3 a_{\tau}(0)}{A}  =   \tfrac{ 1 }{ 1 + p_2 r_0 o_0},
		$$
		$$
		f'_2(\tau) := e^{- \tau } ( o_0  -     a_{\tau}(-\tau) )  \equiv 0 \text{ if } o_0  =  a_{\tau}(-\tau) ,
		$$
		$$
		f'_3(\tau) : = e^{ p_6 \tau}  \Bigl[   \tfrac{p_5 + 1 - p_6 r_0 }{p_4}   - \tfrac{1}{ p_4 +   r^2_0 o^2_0}  \Bigr] \equiv 0
		\text{ if } \tfrac{p_5 + 1 - p_6 r_0 }{p_4} = \tfrac{1}{ p_4 +   r^2_0 o^2_0},
		$$
		then $f_i(\tau) = 0$ for all $\tau \geqslant 0$ provided
		$$
		o_0  =  a_{\tau}(-\tau),\ a_{\tau}(0) = \tfrac{A} { p_3 }\Bigl[1 + p_2 \sqrt{\tfrac{p_4(p_6 r_0 - p_5)}{p_5 + 1 - p_6 r_0}} \Bigr]^{-1},\
		\tfrac{p_5}{p_6} \leqslant r_0 \leqslant \tfrac{p_5+1}{p_6}.
		$$
	\end{proof}
	
	%%%%%%%%%%%%%%%%%%%%
	\section{Discussion}
	Existence of non-negative solutions, uniqueness of a steady state and its stability were analyzed for the minimal model of the HPA in \cite{Vinther}. The analytical difference between the model without delay that we studied and the minimal model is the type of the nonlinearities in the equation that describes production and degradation of the adrenocorticotropic hormone $a(t)$ and in the equation for the density $r(t)$ of glucocorticoid receptor. The nonlinear functions in the model that we analyze depend on the product $o(t)r(t)$  but in the minimal model of the HPA they depend on $o(t)$ only.  We obtained similar results but for more complicated non-linearity terms. We also analyzed stability for all possible cases of parameter ranges. We added Lyapunov stability analysis to show the non-linear stability result and we believe that non-linear stability analysis has never done before for this type of model.  
	
	For the model with delay we obtained a critical time delay value when the originally stable system becomes unstable as a pair of complex eigenvalues crosses the imaginary axis.  To illustrate the dynamics of complex eigenvalues with respect to time delay we used intersections of zero level sets between real and imaginary parts of the characteristic equation as a first approximation of a complex eigenvalue and after that we used Newton iterations to improve the accuracy.  We believe that stability analysis for HPA model with respect to time delay was not done before and our results are new in this area. 
	
	For certain parameter values we rigorously proved existence of non-negative periodic solutions for a given period  $T$ (under an assumption that a given period does not coincide with the value of time delay) and illustrated different periodic solutions numerically. Our numerical simulations revealed that the period of the solution is very sensitive to small perturbations of parameter values.

	\section{Appendix}
	\begin{lemma}[Routh Hurwitz Criteria for a Nonlinear System] \label{H-0}
		Suppose
		\begin{equation}\label{hur}
		\dot{\textbf{x}} = \textbf{f}(\textbf{x}), \  \textbf{f} : \mathbb{R}^3 \to \mathbb{R}^3, \   \textbf{x}(t_0) = \textbf{x}_0.
		\end{equation}
		Suppose $\textbf{x}_s$ is a fixed point of (\ref{hur}) and the characteristic polynomial at the
		fixed point is
		$$
		\lambda^3 + \alpha_1 \lambda^2 + \alpha_2 \lambda + \alpha_3 = 0, \ \ \alpha_i \in \mathbb{R}^1.
		$$
		If $\alpha_1 > 0,\, \alpha_3 > 0$ and $\alpha_1 \alpha_2 > \alpha_3$, then the fixed point is asymptotically stable. If
		$\alpha_1 < 0,\, \alpha_3 < 0$ or $\alpha_1 \alpha_2 < \alpha_3$, then the fixed point is unstable.
	\end{lemma}
	
	\begin{theorem}\label{delay}
		(see \cite{Cook})
		Consider equation (\ref{eq:98}), where $P(z)$ and $Q(z)$ are analytic functions in a right half-plane $Re(z)>-\delta$, $\delta>0$, which satisfy the following conditions:
		\begin{enumerate}[i)]
			\item $P(z)$ and $Q(z)$ have no common imaginary zeros;
			\item $\overline{P(-iy)}=P(iy)$ and $\overline{Q(-iy)}=Q(iy)$ for real $y$;
			\item $P(0)+Q(0)\neq0$;
			\item There are at most a finite number of roots of (\ref{eq:98}) in the right half-plane when $\tau=0$; 
			\item $F(y)\equiv |P(iy)|^2-|Q(iy)|^2$ for real $y$ has at most a finite number of real zeros. 
		\end{enumerate}
		Under these conditions, the following statements are true.
		\begin{enumerate}[a)]
			\item Suppose that the equation $F(y)=0$ has no positive roots. Then if (\ref{eq:98}) is stable at $\tau=0$ it remains stable for all $\tau\geq 0$, whereas if it is unstable at $\tau=0$ it remains unstable for all $\tau\geq 0$.
			\item Suppose that the equation $F(y)=0$ has at least one positive root and that each positive root is simple. As $\tau$ increases, stability switches may occur. There exists a positive number $\tau_c$ such that the equation (\ref{eq:98}) is unstable for all $\tau>\tau_c$. As $\tau$ varies from $0$ to $\tau_c$, at most a finite number of stability switches may occur.
		\end{enumerate}
	\end{theorem}

	%%%%%%%%%%%%%%%%%%%%%%%%%%%%%%%%%
	\markboth{C. JOHNSON, R. M. TARANETS, M. CHUGUNOVA, N. VASYLYEVA}{\rm EXISTANCE AND STABILITY ANALYSIS OF HPA AXIS MODEL}
	
	%%%%%%%%%%%%%%%%%%%%%%%%%%%%%%


\vspace*{6pt}
	
	\begin{thebibliography}{00}
		\markboth{C. JOHNSON, R. M. TARANETS, M. CHUGUNOVA, N. VASYLYEVA}{\rm EXISTANCE AND STABILITY ANALYSIS OF HPA AXIS MODEL}
		
		%\begin{thebibliography}{00}
		
		
		%\harvarditem{Arnold \textit{et al.}}{2000}{arnoldtwo}
		%\textsc{Arnold, D.N., Falk, R.S. \& Winther, R.} (2000)   Multigrid
		%in $H({\it div})$ and $H(curl)$,   {\em Numer. Math.,}   \textbf{85},
		%197--218.
		\bibitem{Andersen}
		\textsc{Andersen, M., Vinther, F.,\& Ottesen, J.}(2013) Mathematical modeling of the hypothalamic-pituitary-adrenal gland (HPA) axis, including hippocampal mechanisms, {\em Math Biosci,} \textbf{246}(1), 122-138.
		
		\bibitem{Bairagi}
		\textsc{Bairagi, N., Chatterjee, S., \& Chattopadhyay, J.}(2008) Variability in the secretion of corticotropin-releasing hormone, adrenocorticotropic hormone and cortisol and understandability of the hypothalamic-pituitary-adrenal axis dynamics — a mathematical study based on clinical evidence, {\em Mathematical Medicine and Biology,} \textbf{25}(1), 37-63.
		
		\bibitem{Barresi}
		\textsc{Barresi, R., Lombardo, M.C., \& Sammartino, M.}(2015) Hopf bifurcation analysis of the generalized Lorenz system with time delayed feedback control, {\em arXiv:1406.4694}
		
		\bibitem{Conrad}
		\textsc{Conrad, M., Hubold, C., Fischer, B., \& Peters, A.}(2009) Modeling the hypothalamus-pituitary-adrenal system: Homeostasis by interacting positive and negative feedback, {\em Journal of Biological Physics,} \textbf{35}(2), 149-162.
		
		\bibitem{Cook}
		\textsc{Cooke, K.L. \& van den Driessche, P.}(1986) On zeros of some transcendental equations, {\em Funkcial. Ekvac.} \textbf{29}, 77-90.
		
		\bibitem{Gold}
		\textsc{Gold, P.W. \& Chrousos, G.P.}(2002) Organization of the stress system and its dysregulation in melancholic and atypical depression: high vs low CRH/NE states, {\em Mol Psychiatry,} \textbf{7}, 254-75.
		
		\bibitem{Gonzalez}
		\textsc{Gonzalez-Heydrich, J., Steingard, R.J., \& Kohane, I.}(1994)
		A computer simulation of the hypothalamic-pituitary-adrenal axis,{\em Proc Annu Symp Comput Appl Med Care,} \textbf{1010}
		
		\bibitem{Gupta}
		\textsc{Gupta, S., Aslakson, E., Gurbaxani, B.M., \& Vernon, S.D.} (2007)  Inclusion of the glucocorticoid receptor in a hypothalamic pituitary adrenal axis model reveals bistability, {\em Theor. Biol. Med. Model.,}  \textbf{4}, 8.
		
		\bibitem{Hosseinichimeh}
		\textsc{Hosseinichimeh, N., Rahmandad, H., \& Wittenborn, A.}(2015) Modeling the hypothalamus-pituitary-adrenal axis: A review and extension, {\em Mathematical Biosciences,} \textbf{268}, 52-65.
		
		\bibitem{Jelic}
		\textsc{Jelić, S., Čupić, Ž., \& Kolar-Anić, L.}(2005) Mathematical modeling of the hypothalamic-pituitary-adrenal system activity, {\em Mathematical Biosciences,} \textbf{197}(2), 173-187.
		
		\bibitem{Juruena}
		\textsc{Juruena, M.F., Cleare, A.J., \& Pariante, C.M.}(2004) The hypothalamic pituitary adrenal axis, glucocorticoid receptor function and relevance to depression, {\em Rev Bras Psiquiatr,} \textbf{26}, 189–201.
		
		\bibitem{Krasn1}
		\textsc{Krasnosel'skii, M.A.}(1968) The Operator of Translation along the Trajectories
		of Differential Equations, {\em Providence: American Mathematical Society,
			Translation of Mathematical Monographs,} \textbf{19}, 294p.
		
		\bibitem{Krasn2}
		\textsc{Krasnosel'skii, M.A.}(1963) An alternative principle for the existence of periodic solutions for differential equations with retarded argument, {\em Dokl. Akad. Nauk SSSR,} \textbf{152}(4), 801--804.
		
		\bibitem{Lenbury}
		\textsc{Lenbury, Y. \& Pornsawad, P.}(2005) A delay-differential equation model of the feedback-controlled hypothalamus-pituitary-adrenal axis in humans, {\em Math Med Biol,} \textbf{22}, 15-33.
		
		\bibitem{Matlab}
		MATLAB version 7.10.0. Natick, Massachusetts: The MathWorks Inc., 2017.
		
		\bibitem{McEwen}
		\textsc{McEwen, B.S.}(2007) Physiology and neurobiology of stress and adaptation:central role of the brain, {\em Physiol. Rev.,}  \textbf{87}, 873--904.
		
		\bibitem{Rankin1}
		\textsc{Rankin, J., Walker, J.J., Windle, R., Lightman, S.L., \& Terry, J.R.} (2012) Characterizing Dynamic Interactions between Ultradian Glucocorticoid Rhythmicity and Acute Stress Using the Phase Response Curve,  {\em PLoS ONE,} \textbf{7}(2): e30978. doi:10.1371/journal.pone.0030978.
		
		\bibitem{Rohleder}
		\textsc{Rohleder, N., Joksimovic, L., Wolf, J.M., \& Kirschbaum, C.}(2004) Hypocortisolism and increased glucocorticoid sensitivity of pro-inflammatory cytokine production in Bosnian war refugees with posttraumatic stress disorder,{\em Biol Psychiatry,} \textbf{55}, 745-751.
		
		\bibitem{Shampine}
		\textsc{Shampine, L.F. \& Thompson, S.}(2000) Solving Delay Differential Equations with dde23, (Lecture notes)
		
		\bibitem{Sriram}
		\textsc{Sriram, K., Rodriguez-Fernandez, M., \& Doyle, F.J.}(2012) III Modeling cortisol dynamics in the neuroendocrine axis distinguishes normal, depression, and post-traumatic stress disorder (PTSD) in humans,{\em PLoS Computational Biology,} \textbf{8}(2), 1-15.
		
		\bibitem{Vinther}
		\textsc{Vinther, F., Andersen, M., \& Ottesen, J.}(2011) The minimal model of the hypothalamic-pituitary-adrenal axis, {\em Journal of Mathematical Biology,} \textbf{63}(4), 663-690.
		
		\bibitem{Walker}
		\textsc{Walker, J.J., Terry, J.R., \& Lightman, S.L.}(2010) Origin of ultradian pulsatility in the
		hypothalamicpituitary-adrenal axis, {\em Proc. R. Soc. B,}\textbf{277}, 1627--1633.
		
	\end{thebibliography}
\end{document}